\documentclass[10pt]{article}
\usepackage[a4paper,centering]{geometry}
\usepackage{mathtools}
\usepackage{amsthm,amssymb}
\usepackage{mathrsfs}

\newtheorem{prop}{Proposition}[section]
\newtheorem{thm}[prop]{Theorem}
\newtheorem{lemma}[prop]{Lemma}

\theoremstyle{remark}
\newtheorem{rmk}[prop]{Remark}

\newcommand{\sC}{\mathsf{C}}
\newcommand{\dom}{\mathsf{D}}
\newcommand{\E}{\mathop{{}\mathbb{E}}\nolimits}
\newcommand{\cF}{\mathscr{F}}
\newcommand{\cL}{\mathscr{L}}
\renewcommand{\P}{\mathbb{P}}
\newcommand{\erre}{\mathbb{R}}

\numberwithin{equation}{section}

\DeclarePairedDelimiter\abs{\lvert}{\rvert}
\DeclarePairedDelimiter\norm{\lVert}{\rVert}
\DeclarePairedDelimiterX\ip[2]{\langle}{\rangle}{#1,#2}

\newcommand{\embed}{\hookrightarrow}

\begin{document}
\title{Singular perturbations and asymptotic expansions for SPDEs with
  an application to term structure models}

\author{Sergio Albeverio\thanks{Institut f\"ur Angewandte Mathematik,
    Universit\"at Bonn, Endenicher Allee 60, D-53115 Bonn,
    Germany. E-mail: \texttt{albeverio@iam.uni-bonn.de}}%
  \and Carlo Marinelli\thanks{Department of Mathematics, University
    College London, Gower Street, London WC1E 6BT, UK. URL:
    \texttt{http://goo.gl/4GKJP}}%
  \and Elisa Mastrogiacomo\thanks{Dipartimento di Economia,
    Universit\`a degli Studi dell'Insubria, via Montegeneroso, 71,
    21100 Varese, Italy. E-mail:
    \texttt{elisa.mastrogiacomo@uninsubria.it}}}

\date{\normalsize December 28, 2020}

\maketitle

\begin{abstract}
  We study the dependence of mild solutions to linear stochastic
  evolution equations on Hilbert space driven by Wiener noise, with
  drift having linear part of the type $A+\varepsilon G$, on the
  parameter $\varepsilon$. In particular, we study the limit and the
  asymptotic expansions in powers of $\varepsilon$ of these solutions,
  as well as of functionals thereof, as $\varepsilon \to 0$, with good
  control on the remainder.  These convergence and series expansion
  results are then applied to a parabolic perturbation of the Musiela
  SPDE of mathematical finance modeling the dynamics of forward rates.
\end{abstract}


\section{Introduction}
Consider the family of stochastic evolution equations
\begin{equation}
\label{eq:ep}
du_\varepsilon = (A+\varepsilon G)u_\varepsilon\,dt + \alpha\,dt +
B\,dW, \qquad u_\varepsilon(0)=u_0,
\end{equation}
set in a Hilbert space $H$ and indexed by $\varepsilon>0$, where $A$
and $G$ are linear maximal dissipative operators on $H$ such that
$A+\varepsilon G$ is also maximal dissipative, $\alpha$ and $B$ are
coefficients satisfying suitable measurability, integrability and
regularity conditions, and $W$ is a cylindrical Wiener
process. Precise assumptions on the data of the problem are given
below. 

Our main goal is to obtain an expansion of the difference
$u_\varepsilon-u$ in a (finite) series in powers of $\varepsilon$ plus
a remainder term, where $u_\varepsilon$ and $u$ are the unique mild
solutions to \eqref{eq:ep} with $\varepsilon>0$ and $\varepsilon=0$,
respectively.
Results in this sense are obtained assuming that the semigroups
generated by $A$ and $G$ commute.  As a first step, we show that
$u_\varepsilon$ converges to $u$ as $\varepsilon \to 0$, also in the
case where $\alpha$ and $B$ are (random, time-dependent) Lipschitz
continuous functions of the unknown, in suitable norms implying the
convergence in probability uniformly on compact intervals in time. For
such convergence result to hold it is enough that the resolvent of
$A+\varepsilon G$ converges to the resolvent of $A$ as $\varepsilon
\to 0$ in the strong operator topology, without any commutativity
assumption. Sufficient conditions for the convergence of operators in
the strong resolvent sense have been largely studied (see, e.g.,
\cite{EnNa,Oka:sing} and references therein) and can be readily
applied to obtain convergence results for solutions to stochastic
evolution equations. On the other hand, expansions in power series of
$u_\varepsilon-u$ are considerably harder to obtain. In fact, it is
well known that solutions to singularly perturbed equations, also in
the simpler setting of deterministic ODEs, do \emph{not} admit series
expansions in the perturbation parameter. This phenomenon appears also
in the class of stochastic equations studied here, as it is quite
obvious. This is essentially the reason behind the commutation
assumption on the semigroups generated by $A$ and $G$, as well as on
the regularity conditions on the initial datum $u_0$ and on the
coefficients $\alpha$ and $B$ (see \S\ref{sec:exp} below, where
asymptotic expansion results are obtained also for functionals of
$u_\varepsilon$).

As an application of the abstract results, we consider a singularly
perturbed transport equation on $\erre$ where, roughly speaking, $A$
and $G$ are the first and second derivative, respectively. This
equation can be seen as a singular perturbation of an extension of
Musiela's SPDE from a weighted Sobolev space on $\erre_+$ to the
corresponding one on $\erre$. The motivation for considering this
problem comes from the interesting article \cite{Cont:TSIR}, where the
author argues that second-order parabolic SPDEs reproduce many
stylized statistical properties of forward curves. On the other hand,
if forward rates satisfy a Heath-Jarrow-Morton dynamics, the
differential operator in the drift of the corresponding SPDE must be
of first order. It is then natural to consider singular perturbations
of the (first-order) Musiela SPDE by second-order differential
operators and to look for conditions implying uniform convergence of
the ``perturbed'' forward rates, as well as of implied bond prices, to
the corresponding ``unperturbed'' forward rates and bond prices, as
well as a more precise description of the dependence of the pricing
error on the ``size'' of the perturbation. Results in this sense are
obtained in the form of asymptoptic expansions in $\varepsilon$ of the
solution $u_\varepsilon$ to a second-order perturbation of a suitable
extension of the Musiela SPDE, as well as of functionals thereof.

The rest of the text is organized as follows. In \S\ref{sec:prel} we
introduce some notation, we recall basic results from semigroup
theory, and we establish some inequalities and identities for classes
of stochastic convolutions. In \S\ref{sec:3} we show that a
commutation assumption between the semigroups generated by $A$ and $G$
implies that the closure of $A+\varepsilon G$ converges to $A$ in the
strong resolvent sense as $\varepsilon \to 0$. This allows, thanks to
a general convergence result for mild solutions to stochastic
evolution equations, to deduce the convergence of $u_\varepsilon$ to
$u$ in a suitable norm. Under further regularity assumptions on $u_0$,
$\alpha$ and $B$, expansions of the difference $u_\varepsilon-u$ and
of functionals thereof as power series in $\varepsilon$ are obtained
in \S\ref{sec:exp}, which is the core of the work. Finally, the
applications described above to Musiela's SPDE are developed in
\S\ref{sec:Mus}.

\medskip

\noindent
\textbf{Acknowledgments.} The first-named author would like to thank
the the Economics Department of the Universit\`a dell'Insubria, Varese
and the Universit\`a degli Studi di Milano for warm hospitality.
Large part of the work for this paper has been carried out while the
second-named author was visiting the Interdisziplin\"ares Zentrum
f\"ur Komplexe Systeme at the University of Bonn.
The third-named author would like to thank the Institute of Applied
Mathematics and the Hausdorff Center for Mathematics at the University
of Bonn for warm hospitality.


\section{Preliminaries}
\label{sec:prel}
Throughout this section we shall use $E$ and $F$ to denote two Banach
spaces. The domain of a linear operator $L$ with graph in $E \times F$
will be denoted by $\dom(L)$. The Banach space of continuous
$k$-linear operators from $E^k$ to $F$, $k \in \mathbb{N}$, is denoted
by $\cL_k(E;F)$ (without subscript, as usual, if $k=1$). Given $h \in
E$ and $k \in \mathbb{N}$, we shall set $h^{\otimes k}=(h,\ldots,h)
\in E^k$.  If $E$ and $F$ are Hilbert spaces, $\cL^2(E;F)$ will stand
for the Hilbert space of Hilbert-Schmidt operators from $E$ to $F$. An
expression of the type $a \lesssim b$ means that there exists a
positive constant $N$ such that $a \leq Nb$, and $a \eqsim b$ stands
for $a \lesssim b$ and $b \lesssim a$.

We recall the following form of Taylor's formula (see, e.g.,
\cite[p.~349]{Lang:RFA}). Let $U \subseteq E$ be open,
$f \in C^m(U;F)$, $x \in U$ and $h \in E$ such that the segment
$[x,x+h]$ is contained in $U$. Then
\[
  f(x+h) = \sum_{k=0}^{m-1} \frac{1}{k!} D^kf(x)h^{\otimes k}
  + \int_0^1 \frac{(1-t)^{m-1}}{(m-1)!} D^mf(x+th)h^{\otimes m}\,dt.
\]

\medskip

For the purposes of this section only, we denote a strongly continuous
semigroup on a Hilbert space $H$ and its generator by $S$ and $A$,
respectively. As is well known, there exist $M \geq 1$ and
$w \in \erre$ such that $\norm{S(t)} \leq Me^{wt}$
for all $t \geq 0$.
Let $m \geq 1$ be an integer. If $\phi \in \dom(A^m)$, one has the
Taylor-like formula
\[
S(t)\phi = \sum_{k=0}^{m-1} \frac{t^k}{k!} A^k\phi
+ \frac{1}{(m-1)!} \int_0^t (t-u)^{m-1}S(u)A^m\phi\,du
\]
(see, e.g., \cite[Proposition~1.1.6]{BuBe}). We recall that $A^m$ is a
closed operator and that $\dom(A^m)$ is a Hilbert space with scalar
product
\[
\ip[\big]{\phi}{\psi}_{\dom(A^m)} = \ip[\big]{\phi}{\psi} +
\ip[\big]{A\phi}{A\psi} + \cdots + \ip[\big]{A^m\phi}{A^m\psi}.
\]

\smallskip

Let $T$ be a further strongly continuous semigroup on $H$. We shall
say that $S$ and $T$ commute if $S(t)T(t)=T(t)S(t)$ for all
$t \in \erre_+$. It is immediate that the product semigroup $ST$ is
strongly continuous. It also follows that $S(s)T(t)=T(t)S(s)$ for all
$t,s \geq 0$: first one proves it for rational $s$ and $t$, hence the
general case follows by density and continuity. For details see, e.g.,
\cite[p.~44]{EnNa}. Moreover $T$ leaves invariant $\dom(A)$: in fact,
for any $f \in \dom(A)$, one has
\[
  \lim_{h \to 0} \frac{S(h)T(t)f - T(t)f}{h} =
  T(t) \Bigl( \lim_{h \to 0} \frac{S(h)f - f}{h} \Bigr) = T(t)Af.
\]
This also implies, by uniqueness of the limit, that
$T(t)Af=AT(t)f$. These observations in turn imply that the resolvent
$R_\lambda$ of the generator of $T$ commutes with $A$, in the sense
that, if $f \in \dom(A)$, then $R_\lambda f \in \dom(A)$ and
$R_\lambda Af = AR_\lambda f$ (cf.~\cite[p.~171]{Kato}).

\medskip

All stochastic elements will be defined on a fixed probability space
$(\Omega,\cF,\P)$, endowed with a filtration $(\cF_t)_{t \in [0,T]}$,
with $T$ a fixed positive number, that is assumed to satisfy the
so-called usual assumptions. All equalities and inequalities between
random variables are meant to hold outside a set of probability
zero. We shall denote by $W$ a cylindrical Wiener process on a real
separable Hilbert space $U$.
We shall denote the closed subspace of $L^p(\Omega;C([0,T];H))$,
$p>0$, of $H$-valued adapted continuous processes by $\sC^p$, which is
hence a quasi-Banach space itself (with the induced quasi-norm).
Given a progressively measurable process
$C \in L^0(\Omega;L^2(0,T;\cL^2(U;H)))$, the stochastic convolution
$S \diamond C$ is the $H$-valued process defined by
\[
S \diamond C(t) := \int_0^t S(t-s)C(s)\,dW(s) \qquad \forall t \in [0,T],
\]
provided that the stochastic integral exists.  Similarly, if
$f \in L^0(\Omega;L^1(0,T;H))$, we shall define the $H$-valued process
$S \ast f$ by
\[
S \ast f(t) := \int_0^t S(t-s)f(s)\,ds \qquad \forall t \in [0,T].
\]
The stochastic integral of a process $F$ with respect to $W$ will be
occasionally denoted by $F \cdot W$ for typographical convenience.
\begin{lemma}
  \label{lm:XX}
  Let $p>0$, $C \in L^p(\Omega;L^2(0,T;\cL^2(U;H)))$ be a
  progressively measurable process, and $n \geq 0$. One has
  \[
  \E\norm[\bigg]{\int_0^t (t-s)^n S(t-s) C(s)\,dW(s)}^p
  \lesssim M^p \E\biggl( \int_0^t (t-s)^{2n} e^{2w(t-s)}
  \norm[\big]{C(s)}^2_{\cL^2(U;H)}\,ds
  \biggr)^{p/2}.
  \]
\end{lemma}
\begin{proof}
  For any $\delta>0$, one has
  \[
  \norm*{\int_0^t (t-s)^n S(t-s)B(s)\,dW(s)}_{L^p(\Omega;H)} \leq
  \sup_{t_0 \in [t,t+\delta]}%
  \norm*{\int_0^t (t_0-s)^n S(t_0-s) B(s)\,dW(s)}_{L^p(\Omega;H)}.
  \]
  Since $\bigl((t_0-\cdot)S(t_0-\cdot)B\bigr) \cdot W$ is a local
  martingale, the Burkholder-Davis-Gundy inequality (see, e.g.,
  \cite{cm:Expo16}) and the ideal property of Hilbert-Schmidt
  operators 
  yield
  \begin{align*}
  &\sup_{t_0 \in [t,t+\delta]}%
  \norm*{\int_0^t (t_0-s)^n S(t_0-s) C(s)\,dW(s)}_{L^p(\Omega;H)}\\
  &\hspace{3em} \lesssim M \sup_{t_0 \in [t,t+\delta]}%
  \norm[\Big]{(t_0-\cdot)^n e^{w(t_0-\cdot)}
  \norm[\big]{C}_{\cL^2(U;H)}}_{L^p(\Omega;L^2(0,t;H)))}.
  \end{align*}
  Setting
  \[
  \phi_\delta(t) =
  \begin{cases}
    e^{w(t+\delta)},  &\text{if $w \geq 0$},\\
    e^{wt},  &\text{if $w < 0$},
  \end{cases}
  \]
  one has $e^{w(t_0-s)} \leq \phi_\delta(t-s)$ for all $t_0 \in
  [t,t+\delta]$ and $s \in [0,t]$, hence
  \begin{align*}
  &\sup_{t_0 \in [t,t+\delta]}%
  \norm*{\int_0^t (t_0-s)^n S(t_0-s) C(s)\,dW(s)}_{L^p(\Omega;H)}\\
  &\hspace{3em} \leq M \norm[\Big]{(t+\delta-\cdot)^n \phi_\delta(t-\cdot)
  \norm[\big]{C}_{\cL^2(U;H)}}_{L^p(\Omega;L^2(0,t;H)))},
  \end{align*}
  therefore
  \[
  \E\norm[\bigg]{\int_0^t (t-s)^n S(t-s) C(s)\,dW(s)}^p
  \lesssim M^p \E\biggl( \int_0^t (t+\delta-s)^{2n} \phi_\delta(t-s)
  \norm[\big]{C(s)}^2_{\cL^2(U;H)}\,ds \biggr)^{p/2}
  \]
  for all $\delta>0$. Taking the limit as $\delta \to 0$ proves the
  claim.
\end{proof}

The following recursive relation for certain nonlinear stochastic
convolutions will be very useful in the sequel.
\begin{lemma}
  \label{lm:desc}
  Let $C \in L^p(\Omega;L^2(0,T;\cL^2(U;H)))$ be a progressively
  measurable process and define, for every $k \in \erre_+$,
  \[
    \Sigma_k(t) := \int_0^t S(t-s) (t-s)^k C(s)\,dW(s).
  \]
  Then $\Sigma_{k+1} = (k+1) S \ast \Sigma_k$.
\end{lemma}
\begin{proof}
  Let $k \geq 1$. Using the identity
  \[
    (t-s)^k = k \int_s^t (r-s)^{k-1}\,dr,
  \]
  the stochastic Fubini theorem, and the semigroup property, one has,
  for every $t \in [0,T]$,
  \begin{align*}
    \Sigma_k(t)
    &= \int_0^t S(t-s) (t-s)^k C(s)\,dW(s)\\
    &= k \int_0^t S(t-s) \biggl( \int_s^t (r-s)^{k-1}\,dr \biggr) C(s)\,dW(s)\\
    &= k \int_0^t\!\!\int_0^r S(t-s) (r-s)^{k-1} C(s)\,dW(s)\,dr\\
    &= k\int_0^t S(t-r) \int_0^r S(r-s) (r-s)^{k-1} C(s)\,dW(s)\,dr\\
    &= k S \ast \Sigma_{k-1}(t).
      \qedhere
  \end{align*}
\end{proof}


\section{Singular perturbations by commuting semigroups}
\label{sec:3}
Let us consider the stochastic evolution equation on the Hilbert space $H$
\[
du = Au\,dt + \alpha(u)\,dt + B(u)\,dW, \qquad u(0)=u_0,
\]
and the family of stochastic evolution equations on $H$ indexed by a
parameter $\varepsilon \geq 0$
\[
  du_\varepsilon = (A + \varepsilon G)u_\varepsilon\,dt
  + \alpha(u_\varepsilon)\,dt + B(u_\varepsilon)\,dW, \qquad
  u_\varepsilon(0)=u_0,
\]
where (i) $A$ and $G$ are linear maximal dissipative operators on $H$
such that the closure of $A+\varepsilon G$, denoted by the same
symbol, is maximal dissipative as well; (ii) the initial datum $u_0$
belongs to $L^0(\Omega,\cF_0;H)$; (iii) the coefficients
\[
  \alpha \colon \Omega \times [0,T] \times H \longrightarrow H,
  \qquad
  B \colon \Omega \times [0,T] \times H \longrightarrow \cL^2(U;H)
\]
are Lipschitz continuous in the third variable, uniformly with respect
to the other ones, and such that $\alpha(\cdot,\cdot,h)$ and
$B(\cdot,\cdot,h)$ are progressively measurable for every $h \in H$.
It is well known that under these conditions the above stochastic
equations admit unique mild solutions $u$ and $u_\varepsilon$,
respectively, with continuous trajectories. Moreover, if
$u_0 \in L^p(\Omega,\cF_0;H)$ for some $p>0$, then $u$ and
$u_\varepsilon$ belong to $\sC^p$ (see, e.g., \cite[Chapter~7]{DZ92}
for the case $p \geq 2$ and \cite{cm:SIMA18} for the general case).

The aim of this section is to provide sufficient conditions ensuring
that $u_\varepsilon \to u$ in $\sC^p$. We rely on the following
convergence result, which is a minor modification of \cite[Theorem
2.4]{cm:JFA13} (see also \cite{KvN2}).
\begin{thm}
  \label{thm:convsol}
  Let $p \in [1,\infty\mathclose[$,
  $u_0 \in L^p(\Omega,\mathscr{F}_0;H)$. Assume that $A+\varepsilon G$
  converges to $A$ in the strong resolvent sense. Then
  $u_\varepsilon \to u$ in $\sC^p$ as $\varepsilon \to 0$.
\end{thm}
We recall that a sequence of maximal dissipative operators $(L_n)$ is
said to converge to a maximal dissipative operator $L$ in the strong
resolvent sense if $(\lambda-L_n)^{-1}x \to (\lambda-L)^{-1}x$ for all
$x \in H$ and all $\lambda>0$.

The problem of the convergence of $u_\varepsilon$ to $u$ is thus
reduced to finding sufficient conditions for the convergence of
$A+\varepsilon G$ to $A$ in the strong resolvent sense as
$\varepsilon \to 0$. In view of the results on asymptotic expansions
in the next sections, we limit ourselves to the special case where the
semigroups generated by $A$ and $G$, denoted respectively by $S_A$ and
$S_G$, commute.
\begin{lemma}
  \label{lm:rc}
  Assume that $S_A$ and $S_G$ commute, i.e. that
  $S_A(t)S_G(t)=S_G(t)S_A(t)$ for all $t \geq 0$. Then
  $A+\varepsilon G$ converges to $A$ in the strong resolvent sense as
  $\varepsilon \to 0$.
\end{lemma}
\begin{proof}
  One has, for any $\lambda>0$ and $f \in H$,
  \[
    \bigl( \lambda - (A+\varepsilon G) \bigr)^{-1}f = \int_0^\infty
    e^{-\lambda t} S_{A+\varepsilon G}(t)f\,dt = \int_0^\infty e^{-\lambda
      t} S_A(t) S_{\varepsilon G}(t)f\,dt
  \]
  and $S_{\varepsilon G}(t)f \to f$ as $\varepsilon \to 0$, hence, by
  dominated convergence,
  \[
    \lim_{\varepsilon \to 0} \bigl( \lambda - (A+\varepsilon G) \bigr)^{-1}f
    = \int_0^\infty e^{-\lambda t} S_A(t)f \,dt = (\lambda - A)^{-1}f.
  \qedhere
  \]
\end{proof}

\begin{rmk}
  (i) Under the assumption that $S_A$ and $S_G$ commute,
  $\dom(A) \cap \dom(G)$ is a core for the generator of the product
  semigroup $S_AS_{\varepsilon G}$, which is contractive and strongly
  continuous. Its generator is hence equal to the closure of
  $A+\varepsilon G$. So the hypothesis of maximal dissipativity of
  (the closure of) $A+\varepsilon G$ is automatically satisfied here.

  \noindent (ii) It is clear from the proof of the previous lemma that
  not even the assumption of dissipativity of $A$ and $G$ is needed,
  but just that the resolvent sets of $A$ and $G$ have non-empty
  intersection. In particular, the statement of the lemma continues to
  hold if $A$ and $G$ are maximal quasi-dissipative, i.e. if there
  exist $a$ and $b \in \erre_+$ such that $A-aI$ and $G-bI$ are
  maximal dissipative. In this respect, as long as one is concerned
  with applications to the stochastic equation, there is no loss of
  generality assuming that $A$ and $G$ are dissipative rather than
  quasi-dissipative, because the latter case reduces to the former by
  adding a linear term to the drift $\alpha$.

  \noindent (iii) The perturbation result in Lemma~\ref{lm:rc}
  strongly relies on the commutativity assumption between $S_A$ and
  $S_G$. This assumption is essential also to derive the asymptotic
  expansion results in the next sections. For other assumptions on $A$
  and $G$ leading to convergence of $A+\varepsilon G$ to $A$ in the
  strong resolvent sense as $\varepsilon \to 0$, see, e.g.,
  \cite{Oka:sing} and references therein.
\end{rmk}


\section{Asymptotic expansion of $u_\varepsilon$}
\label{sec:exp}
Our next goal is to obtain an expression of the difference
$u_\varepsilon-u$ as a finite power series in $\varepsilon$ plus a
remainder. Once such an expression is obtained, the main issue is to
prove estimates on the coefficients of the power series and on the
remainder. Such estimates will crucially depend on suitable regularity
assumptions on the coefficients $\alpha$ and $B$ that will be assumed
throughout the section to be random and time-dependent, but not
explicitly dependent on $u$. In particular, let us consider the
stochastic evolution equations
\begin{equation}
  \label{eq:0a}
  du = Au\,dt + \alpha\,dt + B\,dW, \qquad u(0)=u_0,
\end{equation}
and
\begin{equation}
  \label{eq:epsa}
  du_\varepsilon = Au_\varepsilon\,dt + \varepsilon Gu_\varepsilon\,dt
  + \alpha\,dt + B\,dW, \qquad u_\varepsilon(0)=u_0,
\end{equation}
where $A$ and $G$ are maximal dissipative and generate commuting
semigroups. As before, we denote the closure of $A+\varepsilon G$,
$\varepsilon>0$, by the same symbol.
Moreover, we assume that there exist $p \in [1,\infty\mathclose[$ and
an integer $m \geq 1$ such that
\begin{gather*}
  u_0 \in L^p(\Omega;\dom(G^m)), \qquad
  \alpha \in L^p(\Omega;L^1(0,T;\dom(G^m))),\\
  B \in L^p(\Omega;L^2(0,T;\cL^2(U;\dom(G^m)))).
\end{gather*}
Then equations \eqref{eq:0a} and \eqref{eq:epsa} admit unique mild
solutions $u$ and $u_\varepsilon$ in $\sC^p$,
respectively.\footnote{Since $\alpha$ and $B$ are not functions of the
  unknown, the only non-trivial issue is the pathwise continuity,
  which follows by the contractivity of the semigroups generated by
  $A$ and (the closure of) $A+\varepsilon G$ (see, e.g.,
  \cite[{\S}6.2]{DZ92}).}
Just for convenience, we also assume that $\varepsilon \in [0,1]$.

All results in this section do not use the assumption that $A$ and $G$
are maximal dissipative, except in an indirect way in
Proposition~\ref{prop:ud}, namely through
Theorem~\ref{thm:convsol}. In particular, all results except
Proposition~\ref{prop:ud} continue to hold under the same assumptions
on $u_0$, $\alpha$ and $B$, commutativity of $S_A$ and $S_G$, and the
existence of of unique solutions $u$ and $u_\varepsilon \in \sC^p$ to
\eqref{eq:0a} and \eqref{eq:epsa}, respectively.

\medskip

We begin with a decomposition of $u_\varepsilon$ that is essentially
of algebraic nature.
\begin{prop}
  \label{prop:ina}
  There exist adapted processes $v_1,\ldots,v_{m-1}$ and
  $R_{m,\varepsilon}$ in $\sC^p$ such that
  \[
  u_\varepsilon = u + \sum_{k=1}^{m-1} \frac{\varepsilon^k}{k!} v_k 
  + R_{m,\varepsilon} \qquad \forall \varepsilon \in
  \mathopen]0,1].
  \]
\end{prop}
\begin{proof}
  It follows by commutativity of $S_A$ and $S_G$ that
  \[
  u_\varepsilon = S_AS_{\varepsilon G}u_0 %
  + S_AS_{\varepsilon G} \ast \alpha %
  + S_AS_{\varepsilon G} \diamond B,
  \]
  where, by the Taylor-like formula for strongly continuous semigroups
  of \S\ref{sec:prel},
  \[
  S_{\varepsilon G}(t) = \sum_{k=0}^{m-1} \frac{t^k}{k!} \varepsilon^k G^k
  + \frac{\varepsilon^m}{(m-1)!} \int_0^t (t-r)^{m-1} S_{\varepsilon G}(r) G^m,
  \]
  as an identity of linear operators on $\dom(G^m)$.
  Since $u_0 \in L^p(\Omega;\dom(G^m))$, one has
  \[
  S_{A+\varepsilon G}(t)u_0 = \sum_{k=0}^{m-1} 
  \frac{\varepsilon^k t^k}{k!} S_A(t) G^ku_0
  + \frac{\varepsilon^m}{(m-1)!} S_A(t)
  \int_0^t (t-r)^{m-1}S_{\varepsilon G}(r)G^m u_0\,dr.
  \]
  Similarly, since $\alpha \in L^p(\Omega;L^1(0,T;\dom(G^m)))$ and
  $B \in L^p(\Omega;L^2(0,T;\cL^2(U;\dom(G^m))))$,
  \begin{align*}
    &S_{A+\varepsilon G} \ast \alpha(t) - \sum_{k=0}^{m-1} \int_0^t
    \frac{\varepsilon^k(t-s)^k}{k!}
    S_A(t-s)G^k\alpha(s)\,ds\\
    &\hspace{3em} = \frac{\varepsilon^m}{(m-1)!} \int_0^t
    S_A(t-s)\int_0^{t-s} (t-s-r)^{m-1} S_{\varepsilon G}(r) G^m
    \alpha(s)\,dr\,ds
  \end{align*}
  as well as
  \begin{equation}
  \label{eq:rint}
  \begin{split}
    &S_{A+\varepsilon G} \diamond B(t)
    - \sum_{k=0}^{m-1} \int_0^t \frac{\varepsilon^k(t-s)^k}{k!} 
    S_A(t-s) G^k B(s)\,dW(s)\\
    &\hspace{3em} = \frac{\varepsilon^m}{(m-1)!}
    \int_0^t S_A(t-s) \int_0^{t-s} (t-s-r)^{m-1} S_{\varepsilon G}(r) G^m
    B(s)\,dr\,dW(s).
  \end{split}
  \end{equation}
  Therefore, introducing the $H$-valued processes $v_1,\ldots,v_{m-1}$
  defined as
  \begin{align*}
    v_k(t) &:= t^k S_A(t) G^k u_0 %
    + \int_0^t (t-s)^k S_A(t-s) G^k \alpha(s)\,ds\\
    &\phantom{:=\ } + \int_0^t (t-s)^k S_A(t-s) G^k B(s)\,dW(s)
  \end{align*}
  for each $k \in \{1,\ldots,m-1\}$, and the family of $H$-valued
  processes $(R_{m,\varepsilon})_{\varepsilon\in\mathopen]0,1]}$
  defined as
  \begin{align*}
    R_{m,\varepsilon}(t)
    &:= \frac{\varepsilon^m}{(m-1)!} S_A(t)
    \int_0^t (t-r)^{m-1}S_{\varepsilon G}(r)G^m u_0\,dr\\
    &\phantom{:=\ } + \frac{\varepsilon^m}{(m-1)!} \int_0^t
    S_A(t-s)\int_0^{t-s} (t-s-r)^{m-1} S_{\varepsilon G}(r) G^m
    \alpha(s)\,dr\,ds\\
    &\phantom{:=\ } + \frac{\varepsilon^m}{(m-1)!}
    \int_0^t S_A(t-s) \int_0^{t-s} (t-s-r)^{m-1} S_{\varepsilon G}(r) G^m
    B(s)\,dr\,dW(s),
  \end{align*}
  the desired decomposition follows. For each $k \in
  \{0,\ldots,m-1\}$, it is clear that $v_k$ is adapted, so it remains
  to check that $v_k$ belongs to $\sC^p$. For $m=1$ one has
  \[
  v_1 = \frac{1}{\varepsilon} \bigl( u - u_\varepsilon - R_{m,\varepsilon}
  \bigr),
  \]
  hence the claim for $m=1$ follows, choosing
  $\varepsilon \in \mathopen]0,1]$ arbitrarily, by the assumptions on
  $u$ and $u_\varepsilon$, and Lemma~\ref{lm:trep} below. By induction
  on $m$, the claim for general $m$ is proved.
\end{proof}

\begin{rmk}
  Estimates for the $\sC^p$-norm of $v_k$ can be obtained in a more
  direct (and precise) way exploiting the dissipativity of $A$. In
  fact, one has
  \[
  \norm[\big]{(\cdot)^k S_A G^ku_0}_{\sC^p} \leq M_A T^k e^{w_AT}
  \norm[\big]{u_0}_{L^p(\Omega;\dom(G^k))},
  \]
  and
  \begin{align*}
  \norm[\bigg]{\int_0^\cdot (\cdot-s)^k
    S_A(\cdot-s)G^k\alpha(s)\,ds}_{\sC^p}
  \leq M_A T^k e^{w_AT} \norm[\big]{\alpha}_{L^p(\Omega;L^1(0,T;\dom(G^k)))},
  \end{align*}
  (in fact just assuming that $A$ is only the generator of a strongly
  continuous semigroup), as well as, by maximal estimates for
  stochastic convolutions,
  \[
  \norm[\bigg]{\int_0^\cdot (\cdot-s)^k S_A(\cdot-s) G^k
    B(s)\,dW(s)}_{\sC^p} \lesssim T^k e^{w_AT}
  \norm[\big]{B}_{L^p(\Omega;L^2(0,T;\cL^2(U;\dom(G^k))))}.
  \]
  Alternative assumptions on $A$ yield similar estimates for the
  stochastic convolution, for instance if $A$ generates an analytic
  semigroup. We shall not pursue this issue here.
\end{rmk}

We are going to estimate $R_{m,\varepsilon}$ in $L^p(\Omega;H)$
pointwise with respect to the time variable as well as in $\sC^p$. As
mentioned above, such estimates do not use the dissipativity of $A$
and $G$. For this reason, we shall prove them under the sole
assumption that $A$ and $G$ are generators of strongly continuous
semigroups $S_A$ and $S_G$, respectively, with
$\norm{S_A(t)} \leq M_A e^{w_At}$ and
$\norm{S_G(t)} \leq M_G e^{w_Gt}$ for all $t \in \erre_+$, where
$M_A$, $M_G \geq 1$ and $w_A$, $w_G \in \erre$.

Let us set
\begin{align*}
  R^1_{m,\varepsilon}(t) &:= S_A(t)
  \int_0^t (t-r)^{m-1}S_{\varepsilon G}(r)G^m u_0\,dr,\\
  R^2_{m,\varepsilon}(t) &:= \int_0^t S_A(t-s) \int_0^{t-s}
  (t-s-r)^{m-1} S_{\varepsilon G}(r) G^m \alpha(s)\,dr\,ds,\\
  R^3_{m,\varepsilon}(t) &:= \int_0^t S_A(t-s) \int_0^{t-s}
  (t-s-r)^{m-1} S_{\varepsilon G}(r) G^m B(s)\,dr\,dW(s),
\end{align*}
so that
\begin{equation}
  \label{eq:rm}
  R_{m,\varepsilon} = \frac{\varepsilon^m}{(m-1)!} \bigl(%
  R^1_{m,\varepsilon} + R^2_{m,\varepsilon} + R^3_{m,\varepsilon} \bigr),
\end{equation}
and introduce the function $f_\varepsilon\colon \erre_+ \to \erre_+$
defined as
\begin{equation}
  \label{eq:fe}
f_\varepsilon(t) := e^{w_A t} \int_0^{t} (t-r)^{m-1} e^{\varepsilon
  w_G r}\,dr.
\end{equation}
\begin{lemma}
  \label{lm:ino}
  One has, for every $t \in [0,T]$ and $\varepsilon \in [0,1]$,
  \[
  \norm{R^1_{m,\varepsilon}(t)} \leq M_AM_G f_\varepsilon(t)
  \norm[\big]{u_0}_{\dom(G^m)}
  \]
  and
  \[
  \norm{R^2_{m,\varepsilon}(t)} \leq M_AM_G \int_0^t
    f_\varepsilon(t-s) \norm[\big]{\alpha(s)}_{\dom(G^m)}\,ds.
  \]
\end{lemma}
\begin{proof}
  Both estimates are immediate consequences of Minkowski's inequality
  and the definition of $f$. For instance, the second one is given by
  \begin{align*}
    \norm{R^2_{m,\varepsilon}(t)} &\leq M_AM_G \int_0^t e^{w_A(t-s)}
    \norm[\big]{\alpha(s)}_{\dom(G^m)} \int_0^{t-s}
    (t-s-r)^{m-1} e^{\varepsilon w_G r}\,dr\,ds\\
    &= M_AM_G \int_0^t f_\varepsilon(t-s)
    \norm[\big]{\alpha(s)}_{\dom(G^m)}\,ds. \qedhere
  \end{align*}
\end{proof}

The running maximum of the function $f_\varepsilon$ defined in
\eqref{eq:fe} will be denoted by $f^*_\varepsilon$,
i.e. $f^*_\varepsilon(t) := \max_{s \in [0,t]} f_\varepsilon(s)$.
\begin{lemma}
  One has, for every $t \in [0,T]$ and $\varepsilon \in [0,1]$,
  \[
  \norm[\big]{R^1_{m,\varepsilon}(t)}_{L^p(\Omega;H)} \leq M_AM_G
  f_\varepsilon(t) \norm[\big]{u_0}_{L^p(\Omega;\dom(G^m))}
  \]
  and
  \[
  \norm[\big]{R^2_{m,\varepsilon}(t)}_{L^p(\Omega;H)} \leq M_AM_G
  f^*_\varepsilon(t) \norm[\big]{\alpha}_{L^p(\Omega;L^1(0,t;\dom(G^m)))}
  \]
\end{lemma}
\begin{proof}
  The first estimate is evident. The second one follows by
  \begin{align*}
  \int_0^t f_\varepsilon(t-s) \norm[\big]{\alpha(s)}_{\dom(G^m)}\,ds
  &\leq \norm[\big]{f_\varepsilon(t-\cdot)}_{L^\infty(0,t)}
  \norm[\big]{\alpha}_{L^1(0,t;\dom(G^m))}\\
  &= \norm[\big]{f_\varepsilon}_{L^\infty(0,t)}
  \norm[\big]{\alpha}_{L^1(0,t;\dom(G^m))}. \qedhere
\end{align*}
\end{proof}

\begin{lemma}
\label{lm:ri}
  One has, for every $\varepsilon \in [0,1]$,
  \[
  \norm[\big]{R^1_{m,\varepsilon}}_{\sC^p} \leq M_AM_G
  f^*_\varepsilon(T) \norm[\big]{u_0}_{L^p(\Omega;\dom(G^m))}
  \]
  and
  \[
  \norm[\big]{R^2_{m,\varepsilon}}_{\sC^p} \leq M_AM_G
  f^*_\varepsilon(T) \norm[\big]{\alpha}_{L^p(\Omega;L^1(0,T;\dom(G^m)))}
  \]
\end{lemma}
\begin{proof}
  The first estimate is again evident, by Lemma~\ref{lm:ino}. The
  second one follows by
  \[
  \int_0^t f_\varepsilon(t-s) \norm[\big]{\alpha(s)}_{\dom(G^m)}\,ds
  \leq f^*_\varepsilon(t) \norm[\big]{\alpha}_{L^1(0,t;\dom(G^m))}
  \leq f^*_\varepsilon(T) \norm[\big]{\alpha}_{L^1(0,T;\dom(G^m))}.
  \qedhere
  \]
\end{proof}

The estimates of $R^3_{m,\varepsilon}$ are more delicate. The reason
is that the double integral in \eqref{eq:rint} is not a stochastic
convolution. In fact, while it can be written as
\[
\int_0^t R(t-s)B(s)\,dW(s), \qquad R(t) := \int_0^t (t-r)^{m-1}
S_{\varepsilon G}(r) G^m\,dr,
\]
the family of operators $(R(t))_{t\in\erre_+}$ is not a
semigroup. Unfortunately we are not aware of any maximal inequalities
for such ``nonlinear'' stochastic convolutions. We shall nonetheless
obtain estimates on the remainder term $R^3_{m,\varepsilon}$ by
different arguments.
\begin{lemma}
  One has, for every $t \in [0,T]$ and $\varepsilon \in [0,1]$,
  \[
  \norm[\big]{R^3_{m,\varepsilon}(t)}_{L^p(\Omega;H)} \leq M_AM_G
  f^*_\varepsilon(t)
  \norm[\big]{B}_{L^p(\Omega;L^2(0,t;\cL^2(U;\dom(G^m))))}.
  \]
\end{lemma}
\begin{proof}
  We shall use an argument analogous to the one used in the proof of
  Lemma~\ref{lm:XX}. Write
  \[
  C(t,s) := S_A(t-s) \int_0^{t-s} (t-s-r)^{m-1} S_{\varepsilon G}(r)
  G^mB(s)\,dr,
  \]
  so that $R^3_{m,\varepsilon}(t) = \bigl( C(t,\cdot) \cdot W
  \bigr)_t$. Then
  \[
  \norm[\big]{R^3_{m,\varepsilon}(t)}_{L^p(\Omega;H)} \leq
  \sup_{t_0\in[t,t+\delta]} \norm[\big]{\bigl( C(t_0,\cdot) \cdot W
    \bigr)_t}_{L^p(\Omega;H)},
  \]
  where
  \[
  \norm[\big]{\bigl( C(t_0,\cdot) \cdot W \bigr)_t}_{L^p(\Omega;H)}
  \lesssim \norm[\big]{C(t_0,\cdot)}_{L^p(\Omega;L^2(0,t;\cL^2(U;H)))}
  \]
  and
  \begin{align*}
    \norm[\big]{C(t_0,s)}_{\cL^2(U;H)} &\leq M_AM_G e^{w_A(t_0-s)}
    \int_0^{t_0-s} (t_0-s-r)^{m-1} e^{\varepsilon w_G r}
    \norm[\big]{B(s)}_{\cL^2(U;\dom(G^m))}\,dr,
  \end{align*}
  where
  \begin{gather*}
    \exp(w_A(t_0-s)) \leq \exp\bigl(w_A(t+\delta 1_{\{w_A \geq 0\}}-s)\bigr),\\
    \int_0^{t_0-s} (t_0-s-r)^{m-1} e^{\varepsilon w_G r}\,dr \leq
    \int_0^{t+\delta-s} (t+\delta-s-r)^{m-1} e^{\varepsilon w_G r}\,dr
  \end{gather*}
  hence
  \begin{align*}
  \norm[\big]{C(t_0,s)}_{\cL^2(U;H)} &\leq M_AM_G
  \norm[\big]{B(s)}_{\cL^2(U;\dom(G^m))} \, \cdot\\
  &\quad \cdot \exp\bigl(w_A(t+\delta 1_{\{w_A \geq 0\}}-s)\bigr)
   \int_0^{t+\delta-s} (t+\delta-s-r)^{m-1} e^{\varepsilon w_G r}\,dr
  \end{align*}
  for all $t_0 \in [t,t+\delta]$. In particular,
  \begin{align*}
  \sup_{t_0 \in [t,t+\delta]} \norm[\big]{C(t_0,s)}_{\cL^2(U;H)}
  &\leq M_AM_G \norm[\big]{B(s)}_{\cL^2(U;\dom(G^m))} \, \cdot\\
  &\quad \cdot \exp\bigl(w_A(t+\delta 1_{\{w_A \geq 0\}}-s)\bigr)
   \int_0^{t+\delta-s} (t+\delta-s-r)^{m-1} e^{\varepsilon w_G r}\,dr.
  \end{align*}
  Moreover, setting
  \[
  f_{\varepsilon,\delta}(t) := \exp\bigl(w_A(t+\delta 1_{\{w_A \geq
    0\}})\bigr) \int_0^{t+\delta-s} (t+\delta-r)^{m-1} e^{\varepsilon
    w_G r}\,dr,
  \]
  we can write
  \begin{align*}
    \norm[\big]{R^3_{m,\varepsilon}(t)}_{L^p(\Omega;H)} &\leq
    \sup_{t_0\in[t,t+\delta]} \norm[\big]{\bigl( C(t_0,\cdot) \cdot W
      \bigr)_t}_{L^p(\Omega;H)}\\
    &\lesssim \sup_{t_0\in[t,t+\delta]}
    \norm[\big]{C(t_0,\cdot)}_{L^p(\Omega;L^2(0,t;\cL^2(U;H)))}\\
    &\leq \norm[\Big]{ \sup_{t_0\in[t,t+\delta]}
      \norm[\big]{C(t_0,\cdot)}_{\cL^2(U;\dom(G^m))}}_{L^p(\Omega;L^2(0,t))}\\
    &\leq M_AM_G \norm[\bigg]{\biggl( \int_0^t f^2_{\varepsilon,\delta}(t-s)
    \norm[\big]{B(s)}^2_{\cL^2(U;\dom(G^m))}\,ds \biggr)^{1/2}}_{L^p(\Omega)}.
  \end{align*}
  Since $\delta>0$ is arbitrary, taking the limit as $\delta \to 0$
  yields
  \begin{align*}
  \norm[\big]{R^3_{m,\varepsilon}(t)}_{L^p(\Omega;H)} &\lesssim M_AM_G
  \norm[\big]{f_\varepsilon(t-\cdot) B}_{L^p(\Omega;L^2(0,t;\cL^2(U;\dom(G^m))))}\\
  &\leq M_AM_G f^*_\varepsilon(t)
  \norm[\big]{B}_{L^p(\Omega;L^2(0,t;\cL^2(U;\dom(G^m))))}. \qedhere
  \end{align*}
\end{proof}

\begin{lemma}
\label{lm:trep}
  Let $p \geq 1$. One has, for every $\varepsilon \in [0,1]$,
  \[
  \norm[\big]{R^3_{m,\varepsilon}}_{\sC^p} \lesssim \frac{T^m}{m}
  M_A^2 M_G (e^{w_A T} \vee 1) \bigl( e^{(w_A+\varepsilon w_G)T} \vee 1 \bigr)
  \norm[\big]{B}_{L^p(\Omega;L^2(0,T;\cL^2(U;\dom(G^m))))}.
  \]
\end{lemma}
\begin{proof}
  Thanks to the stochastic Fubini theorem, $R^3_{m,\varepsilon}(t)$
  can be written as
  \begin{align*}
    &\int_0^t S_{\varepsilon G}(r) \int_0^{t-r}
    (t-r-s)^{m-1} S_A(t-s) G^m B(s)\,dW(s)\,dr\\
    &\hspace{3em} = \int_0^t S_{A+\varepsilon G}(r) \int_0^{t-r}
    (t-r-s)^{m-1} S_A(t-r-s) G^mB(s)\,dW(s)\,dr,
  \end{align*}
  thus also, setting
  \[
  \Phi(t) := \int_0^t (t-s)^{m-1} S_A(t-s) G^m B(s)\,dW(s) \qquad
  \forall t \in [0,T],
  \]
  as $S_{A+\varepsilon G} \ast \Phi(t)$, with
  \begin{align*}
  \norm[\big]{S_{A+\varepsilon G} \ast \Phi(t)} &\leq M_AM_G
  \int_0^t e^{w_A(t-s)} e^{\varepsilon w_G(t-s)} \norm[\big]{\Phi(s)}\,ds\\
  &\leq M_AM_G \bigl( e^{(w_A+\varepsilon w_G)T} \vee 1 \bigr)
  \int_0^T \norm[\big]{\Phi(t)}\,dt.
  \end{align*}
  Minkowski's inequality yields
  \[
  \norm[\big]{S_{A+\varepsilon G} \ast \Phi}_{\sC^p}
  \leq M_AM_G \bigl( e^{(w_A+\varepsilon w_G)T} \vee 1 \bigr)
  \int_0^T \norm[\big]{\Phi(t)}_{L^p(\Omega;H)}\,dt,
  \]
  where, by Lemma~\ref{lm:XX},
  \begin{align*}
    \norm[\big]{\Phi(t)}_{L^p(\Omega;H)} &\lesssim
    M_A \norm[\big]{(t-\cdot)^{m-1} e^{w_A(t-\cdot)}
    \norm{B}_{\cL^2(U;\dom(G^m))}}_{L^p(\Omega;L^2(0,t))}\\
    &\leq M_A t^{m-1}(e^{w_A t} \vee 1)
    \norm[\big]{B}_{L^p(\Omega;L^2(0,t;\cL^2(U;\dom(G^m))))},
  \end{align*}
  hence
\[
\int_0^T \norm[\big]{\Phi(t)}_{L^p(\Omega;H)}\,dt \lesssim
\norm[\big]{B}_{L^p(\Omega;L^2(0,T;\cL^2(U;\dom(G^m))))} (e^{w_A T} \vee 1) 
\int_0^T t^{m-1}\,dt,
\]
therefore
\[
\norm[\big]{R^3_{m,\varepsilon}}_{\sC^p} \lesssim \frac{T^m}{m}
M_A^2 M_G (e^{w_A T} \vee 1) \bigl( e^{(w_A+\varepsilon w_G)T} \vee 1 \bigr)
\norm[\big]{B}_{L^p(\Omega;L^2(0,T;\cL^2(U;\dom(G^m))))}. \qedhere
\]
\end{proof}

Let us denote the constant in the Burkholder-Davis-Gundy inequality
with power $p>0$ by $N_p$.
\begin{thm}
  One has, for every $t \in [0,T]$,
  \begin{align*}
  &\norm[\big]{R_{m,\varepsilon}(t)}_{L^p(\Omega;H)} \leq
  \frac{\varepsilon^m}{(m-1)!} M_AM_G \Bigl(%
  f_\varepsilon(t) \norm[\big]{u_0}_{L^p(\Omega;\dom(G^m))} \\
  &\hspace{5em} + f^*_\varepsilon(t)
  \norm[\big]{\alpha}_{L^p(\Omega;L^1(0,t;\dom(G^m)))} %
  + N_p f^*_\varepsilon(t)
  \norm[\big]{B}_{L^p(\Omega;L^2(0,t;\cL^2(U;\dom(G^m))))} \Bigr).
  \end{align*}
  Moreover, if $p \geq 1$, one has
  \begin{align*}
  &\norm[\big]{R_{m,\varepsilon}}_{\sC^p} \leq
  \frac{\varepsilon^m}{(m-1)!} M_AM_G \Bigl(%
  f^*_\varepsilon(T) \norm[\big]{u_0}_{L^p(\Omega;\dom(G^m))} \\
  &\hspace{5em} + f^*_\varepsilon(T)
  \norm[\big]{\alpha}_{L^p(\Omega;L^1(0,t;\dom(G^m)))} \\
  &\hspace{5em} + N_p \frac{T^m}{m} M_A
  (e^{w_A T} \vee 1) \bigl( e^{(w_A+\varepsilon w_G)T} \vee 1 \bigr)
  \norm[\big]{B}_{L^p(\Omega;L^2(0,T;\cL^2(U;\dom(G^m))))} \Bigr).
  \end{align*}
  In particular,
  \[
  \lim_{\varepsilon \to 0}
  \frac{\norm[\big]{R_{m,\varepsilon}}_{\sC^p}}%
  {\varepsilon^{m-1}} = 0.
  \]
\end{thm}
\begin{proof}
  This is an immediate consequence of the previous propositions and
  lemmas in this section, upon observing that $f_\varepsilon$ and
  $f^*_\varepsilon$ converge pointwise to a finite limit as
  $\varepsilon \to 0$.
\end{proof}

\begin{rmk}
  It seems interesting to remark that, without any dissipativity
  assumption on $A$ and $G$, the previous theorem implies that, as
  soon as $m \geq 1$, one has $u_\varepsilon \to u$ in $\sC^p$ as
  $\varepsilon \to 0$ without the need to appeal to
  Theorem~\ref{thm:convsol}.
\end{rmk}

We are now going to identify the process $v_k$ as the $k$-th
derivative at zero of $\varepsilon \mapsto u_\varepsilon$. We shall
actually prove more than this, namely that $u_\varepsilon$ is $m$
times continuously differentiable with respect to $\varepsilon$.
\begin{prop}
  \label{prop:ud}
  The map $\varphi \colon \varepsilon \mapsto u_\varepsilon$ is of
  class $C^m$ from $[0,1]$ to $\sC^p$, with
  \[
  D^k\varphi(0) = v_k \qquad \forall k \in \{1,\ldots,m-1\}.
  \]
\end{prop}
\begin{proof}
  Let $\varepsilon \in [0,1]$ and $h \in \erre$ be such that
  $\varepsilon+h \in [0,1]$. We begin by establishing first-order
  continuous differentiability. One has
  \begin{align*}
    u_{\varepsilon+h}(t)-u_\varepsilon(t)
    &= S_{A+\varepsilon G}(t) \bigl( S_{hG}(t)u_0-u_0 \bigr)\\
    &\quad + \int_0^t S_{A+\varepsilon G}(t-s) \bigl( S_{hG}(t-s)\alpha(s)
      - \alpha(s) \bigr)\,ds\\
    &\quad + \int_0^t S_{A+\varepsilon G}(t-s) \bigl( S_{hG}(t-s)B(s)
      - B(s) \bigr)\,dW(s),    
  \end{align*}
  where, recalling that $S_{hG} = S_G(h\,\cdot)$,
  \[
    \lim_{h \to 0}\frac{S_{hG}(t)u_0-u_0}{h}
    = t \lim_{h \to 0} \frac{S_G(ht)u_0-u_0}{ht} = tGu_0
  \]
  for every $t \in [0,T]$, hence
  \[
    \lim_{h \to 0} \frac{S_{A+\varepsilon G} \bigl( S_{hG}u_0-u_0 \bigr)}{h}
    = [t \mapsto tGu_0]
  \]
  in $\sC^p$ by dominated convergence.
  Similarly, one has
  \begin{align*}
    \lim_{h \to 0}\frac{S_{hG}(t-s)\alpha(s) - \alpha(s)}{h}
    &= (t-s) \lim_{h \to 0}
      \frac{S_G(h(t-s))\alpha(s) - \alpha(s)}{h(t-s)}\\
    &= (t-s) G\alpha(s)
  \end{align*}
  for all $s$, $t \in [0,T]$ with $s \leq t$, hence, again by
  dominated convergence,
  \begin{align*}
    &\lim_{h \to 0} \frac{1}{h} \int_0^\cdot S_{A+\varepsilon G}(\cdot-s)
    \bigl( S_{hG}(\cdot-s)\alpha(s) - \alpha(s) \bigr)\,ds\\
    &\hspace{5em} = \int_0^\cdot S_{A+\varepsilon G}(\cdot-s) (\cdot-s)
      G\alpha(s)\,ds
  \end{align*}
  in $\sC^p$.
  The stochastic convolution term cannot be treated the same way and
  requires more work. We shall write, for simplicity of notation,
  $S_\varepsilon$ in place of $S_{A+\varepsilon G}$. Introducing the
  processes $y_\varepsilon=y^0_\varepsilon$ and $y^1_\varepsilon$ defined as
  \begin{align*}
    y_{\varepsilon}(t) &:= \int_0^t S_{\varepsilon}(t-s)B(s)\,dW(s),\\
    y^1_{\varepsilon}(t) &:= \int_0^t S_{\varepsilon}(t-s) (t-s) GB(s)\,dW(s),
  \end{align*}
  we need to show that
  \begin{equation}
    \label{eq:yeya}
    \lim_{h \to 0} \frac{y_{\varepsilon+h} - y_\varepsilon}{h} =
    y^1_\varepsilon \qquad \text{in } \sC^p.
  \end{equation}
  Duhamel's formula yields
  \begin{align*}
    y_{\varepsilon+h}(t) 
    &= h\int_0^t S_\varepsilon(t-s) Gy_{\varepsilon+h}(s)\,ds +
    \int_0^t S_{\varepsilon}(t-s) B(s)\,dW(s)\\
    &= h\int_0^t S_\varepsilon(t-s) Gy_{\varepsilon+h}(s)\,ds
    + y_\varepsilon(t),
  \end{align*}
  hence
  \begin{align*}
    \frac{y_{\varepsilon+h}(t) - y_\varepsilon(t)}{h} 
    &= \int_0^t S_{\varepsilon}(t-s) Gy_{\varepsilon+h}(s)\,ds\\
    &= \int_0^t S_{\varepsilon}(t-s)
      \int_0^s S_{\varepsilon+h}(s-r) GB(r)\,dW(r)\,ds.
  \end{align*}
  Since $S_{\varepsilon+h} \diamond GB$ converges to
  $S_{\varepsilon} \diamond GB$ in $\sC^p$ as $h \to 0$ by
  Theorem~\ref{thm:convsol}, it follows by dominated convergence that
  \[
  \lim_{h \to 0} \frac{y_{\varepsilon+h} - y_\varepsilon}{h} =
  \int_0^\cdot S_{\varepsilon}(\cdot - s)
      \int_0^s S_{\varepsilon}(s-r) GB(r)\,dW(r)\,ds
  \qquad \text{in } \sC^p.
  \]
  Moreover, by Lemma~\ref{lm:desc},
  \[
  \int_0^t S_{\varepsilon}(t-s) \int_0^s S_{\varepsilon}(s-r)
  GB(r)\,dW(r)\,ds = \int_0^t S_{\varepsilon}(t-s) (t-s)GB(s)\,dW(s)
  = y^1_\varepsilon(t),
  \]
  thus \eqref{eq:yeya} is proved. Furthermore, it follows by the
  assumptions on $B$ that the same argument also
  yields the stronger statement
  \begin{equation}
    \label{eq:yeg}
    \lim_{h \to 0} \frac{G^jy_{\varepsilon+h}-G^jy_{\varepsilon}}{h}
    = G^jy^1_{\varepsilon} \qquad \text{in } \sC^p
    \qquad \forall 0 \leq j \leq m-1
  \end{equation}
  (with $j$ integer). Let us turn to higher-order derivatives. We
  shall only consider the term involving the stochastic convolution,
  as the terms involving the initial datum and the deterministic
  convolution can be treated in a completely analogous (in fact
  easier) way.
  We need to show that the $k$-derivative of
  $\varepsilon \mapsto y_\varepsilon$, denoted by $y^{(k)}$, satisfies
  \[
    y^{(k)}_\varepsilon(t) = \int_0^t S_{A+\varepsilon G} (t-s)
    (t-s)^k G^k B(s)\,dW(s) =: y^k_\varepsilon(t)
  \]
  for all $k \geq 2$, as the case $k=1$ has just been proved.
  We begin with some preparations. Lemma~\ref{lm:desc} implies that
  \begin{equation}
    \label{eq:recia}
    y^k_\varepsilon
    = k S_\varepsilon \ast Gy^{k-1}_\varepsilon
    = k! \, S_\varepsilon^{\ast k} G^k y_\varepsilon
    = k! \, S_\varepsilon^{\ast k} S_\varepsilon \diamond G^kB
  \end{equation}
  for every $k \in \{1,\ldots,m\}$, where $S_\varepsilon^{\ast k}$
  denotes the operation of $k$ times convolution with $S_\varepsilon$,
  i.e.
  \[
    S_\varepsilon^{\ast 1} \phi := S_\varepsilon \ast \phi, \qquad
    S_\varepsilon^{\ast k} \phi = S_\varepsilon \ast
    (S_\varepsilon^{\ast(k-1)} \phi).
  \]
  It follows by a repeated application of Theorem~\ref{thm:convsol}
  that $G^jy^k_{\varepsilon+h} \to G^jy^k_\varepsilon$ in $\sC^p$ as
  $h \to 0$ for all $j$, $k \in \mathbb{N}$ with $j+k \leq m$.
  We shall now proceed by induction, i.e. we are going to prove that,
  for any $\varepsilon \in [0,1]$,
  $y_\varepsilon^{(k)}=y_\varepsilon^k$ implies
  $y_\varepsilon^{(k+1)}=y_\varepsilon^{k+1}$.
  Since
  $y^k_{\varepsilon+h} = k S_{\varepsilon+h} \ast
  Gy^{k-1}_{\varepsilon+h}$, Duhamel's formula yields, setting
  $z:= y^k_{\varepsilon+h}/k$,
  \[
    z(t) = h \int_0^t S_\varepsilon(t-s) Gz(s)\,ds
    + \int_0^t S_\varepsilon(t-s) Gy_{\varepsilon+h}^{k-1}(s)\,ds,
  \]
  therefore, by the identity
  $y^k_\varepsilon = k S_\varepsilon \ast Gy^{k-1}_\varepsilon$,
  \begin{align*}
    y^k_{\varepsilon+h}(t) - y_\varepsilon^k(t)
    &= k \biggl( h\int_0^t S_\varepsilon(t-s)Gz(s)\,ds\\
    &\hspace{5em} + \int_0^t S_\varepsilon(t-s) \bigl(
      Gy_{\varepsilon+h}^{k-1}(s) - Gy_\varepsilon^{k-1}(s) \bigr)\,ds \biggr)\\
    &= h \int_0^t S_\varepsilon(t-s)Gy^k_{\varepsilon+h}(s)\,ds\\
    &\hspace{5em} + k \int_0^t S_\varepsilon(t-s) \bigl(
      Gy_{\varepsilon+h}^{k-1}(s) - Gy_\varepsilon^{k-1}(s) \bigr)\,ds,
  \end{align*}
  hence
  \begin{align*}
    \frac{y^k_{\varepsilon+h}(t) - y_\varepsilon^k(t)}{h}
    &= \int_0^t S_\varepsilon(t-s)Gy^k_{\varepsilon+h}(s)\,ds\\
    &\hspace{5em} + k \int_0^t S_\varepsilon(t-s) G
      \frac{y_{\varepsilon+h}^{k-1}(s) - y_\varepsilon^{k-1}(s)}{h} \,ds,
  \end{align*}
  where, as discussed above,
  $Gy^k_{\varepsilon+h} \to Gy^k_\varepsilon$ in $\sC^p$ as $h \to 0$,
  so that, by dominated convergence,
  \[
    \lim_{h \to 0} \int_0^\cdot S_\varepsilon(\cdot-s)
    Gy^k_{\varepsilon+h}(s)\,ds = \int_0^\cdot S_\varepsilon(\cdot-s)
    Gy^k_\varepsilon(s)\,ds
  \]
  in $\sC^p$. The inductive assumption means that
  \[
    \lim_{h \to 0} \frac{y^{(k-1)}_{\varepsilon+h}-y^{(k-1)}_{\varepsilon}}{h}
    = \lim_{h \to 0} \frac{y^{k-1}_{\varepsilon+h}-y^{k-1}_{\varepsilon}}{h}
    = y^k_\varepsilon
  \]
  in $\sC^p$ for every $\varepsilon \in [0,1]$. The assumptions on $B$
  and \eqref{eq:recia} imply that the inductive assumption also
  yields, in complete analogy to the argument leading to
  \eqref{eq:yeg}, that
  \[
    \lim_{h \to 0}
    \frac{G^jy^{k-1}_{\varepsilon+h} - G^jy^{k-1}_{\varepsilon}}{h} =
    G^jy^k_\varepsilon
  \]
  for every positive integer $j$ such that $j+k \leq m$.
  Therefore, again by dominated convergence, we have
  \[
    \lim_{h \to 0} \int_0^\cdot S_\varepsilon(\cdot - s)
    G \frac{y_{\varepsilon+h}^{k-1}(s) - y_\varepsilon^{k-1}(s)}{h}\,ds
    = \int_0^\cdot S_\varepsilon(\cdot - s) Gy^{k}_\varepsilon(s)\,ds
  \]
  in $\sC^p$, hence we conclude that
  \[
    \lim_{h \to 0} \frac{y^{(k)}_{\varepsilon+h} - y_\varepsilon^{(k)}}{h}
    = \lim_{h \to 0} \frac{y^k_{\varepsilon+h} - y_\varepsilon^k}{h}\\
    = (k+1) S_\varepsilon \ast Gy_\varepsilon^k
    = y_\varepsilon^{k+1},
  \]
  thus concluding the proof of the induction step.
\end{proof}

\subsection{Asymptotic expansion of functionals of $u_\varepsilon$}
We are now going to consider asymptotic expansions of processes of the
type $F(u_\varepsilon)$, where $F$ is a functional taking values in a
Banach space. All assumptions stated at the beginning of the sections
are still in force.

We begin with a simple case.
\begin{prop}
  Let $E$ be a Banach space and $F\colon \sC^p \to E$ be of class
  $C^{m-1}$, $m \geq 2$. Then there exist $w_1,\ldots,w_{m-2} \in E$
  and $R_{m-1,\varepsilon} \in E$ such that, for every
  $\varepsilon \in \mathopen]0,\delta]$,
  \[
  F(u_\varepsilon) = F(u) + \sum_{n=1}^{m-2} \frac{\varepsilon^n}{n!} w_n +
  R_{m-1,\varepsilon},
  \]
  where
  \begin{equation}
  \label{eq:wn}
  w_n = \sum_{j=1}^n \sum_{\substack{k_1+\dots+k_n=j\\%
      k_1+2k_2+\dots + nk_n =n}} \frac{n!}{k_1! \cdots k_n!}
  D^jF(u) \Bigl( \bigl(v_1/1!\bigr)^{\otimes k_1},\ldots,%
  \bigl(v_n/n!\bigr)^{\otimes k_n} \Bigr)
  \end{equation}
  and
  \[
  \lim_{\varepsilon \to 0} \frac{R_{m-1,\varepsilon}}{\varepsilon^{m-2}} = 0
  \]
\end{prop}
\begin{proof}
  Since $\varepsilon \mapsto u_\varepsilon$ is of class $C^m$ from
  $[0,\delta]$ to $\sC^p$ by Proposition~\ref{prop:ud}, it follows
  that $\varepsilon \mapsto F(u_\varepsilon)$ belongs to
  $C^{m-1}([0,\delta];E)$. The expression for $F(u_\varepsilon)$ then
  follows immediately by Taylor's theorem, and the expression for
  $w_n$ follows by the formula for higher derivatives of composite
  functions (sometimes called Fa\`a di Bruno's formula -- see, e.g.,
  \cite[p.~272]{Bog:TVS}). Furthermore, denoting the map
  $\varepsilon \mapsto u_\varepsilon$ by $\varphi$, one has
  \[
  R_{m-1,\varepsilon} = \int_0^t \frac{(1-t)^{m-2}}{(m-2)!}
  D^{m-1}(F \circ \varphi)(\varepsilon t) \varepsilon^{m-1}\,dt,
  \]
  where $D^{m-1}(F \circ \varphi)$ is bounded in $E$ on the compact
  interval $[0,\delta]$ because it is continuous thereon. Denoting the
  maximum of the $E$-norm of this function on $[0,\delta]$ by
  $M_\delta$, we have
  \[
  \frac{R_{m-1,\varepsilon}}{\varepsilon^{m-2}} \leq \varepsilon
  \frac{M_\delta}{(m-1)!},
  \]
  where the right-hand side obviously tends to zero as $\varepsilon
  \to 0$.
\end{proof}

We shall now assume that $F \in C^m(\sC^p;E)$ and derive an expansion
of $u_\varepsilon-u$ of order $m-1$. Note that an argument based on
Taylor's formula for $\varepsilon \mapsto F(u_\varepsilon)$, as in the
previous proposition, does not work because
$\varepsilon \mapsto u_\varepsilon$ is only of class $C^{m-1}$, hence
its composition with $F$ is also of class $C^{m-1}$. We are going to
use instead a construction based on composition of power series.
\begin{thm}
  \label{thm:ae}
  Let $E$ be a Banach space and $F\colon \sC^p \to E$ be of class
  $C^m$, $m \geq 1$. Then there exist $w_1,\ldots,w_{m-1} \in E$ and
  $R_{m,\varepsilon} \in E$ such that, for every
  $\varepsilon \in \mathopen]0,\delta]$,
  \[
  F(u_\varepsilon) = F(u) + \sum_{n=1}^{m-1} \frac{\varepsilon^n}{n!} w_n +
  R_{m,\varepsilon},
  \]
  where the $(w_n)$ are defined as in \eqref{eq:wn} and
  \[
  \lim_{\varepsilon \to 0} \frac{R_{m,\varepsilon}}{\varepsilon^{m-1}} = 0
  \]
\end{thm}
\begin{proof}
  Taylor's formula applied to $F$ yields
  \begin{equation}
    \label{eq:fufi}
  \begin{split}
    F(u_\varepsilon) - F(u) &= \sum_{n=1}^{m-1} D^nF(u)(u_\varepsilon-u)^{\otimes n}\\
    &\quad + \int_0^t \frac{(1-t)^{m-1}}{(m-1)!}
    D^mF(tu_\varepsilon+(1-t)u)%
    (u_\varepsilon-u)^{\otimes m}\,dt,
  \end{split}
  \end{equation}
  where, by Proposition~\ref{prop:ina},
  \begin{equation}
    \label{eq:ue}
    u_\varepsilon - u = \sum^{m-1}_{k=1} \frac{v_k}{k!} \varepsilon^k +
    R_{m,\varepsilon} = \sum^{m-1}_{k=1} \frac{v_k}{k!} \varepsilon^k +
    \bar{R}_{m,\varepsilon} \varepsilon^m,
  \end{equation}
  where $\bar{R}_{m,\varepsilon} \in \sC^p$ by \eqref{eq:rm} and
  Lemmas~\ref{lm:ri} and \ref{lm:trep}.
  Multilinearity of the higher-order derivatives of $F$ implies that
  \begin{equation}
    \label{eq:dudu}
    \sum_{n=1}^{m-1} D^nF(u)(u_\varepsilon-u)^{\otimes n} = 
    \sum_{n=1}^{m-1} \frac{w_n}{n!} \varepsilon^n 
    + \sum_{n=m}^{(m-1)m} a_n \varepsilon^n,
  \end{equation}
  where $w_n$, $n=1,\ldots,m-1$ are defined as in \eqref{eq:wn}, and
  the $a_n$ are (finite) linear combinations of terms of the type
  \[
  D^jF(u) \Bigl( v_1^{\otimes k_1},\ldots,v_n^{\otimes k_n},%
  \bar{R}_{m,\varepsilon}^{\otimes k_{n+1}} \Bigr),
  \]
  where $j \in \{1,\ldots,m-1\}$ and $k_1,\ldots,k_{n+1} \in
  \mathbb{N}$, $k_1+\cdots+k_n+k_{n+1}=j$.

  Let us show that $w_n \in E$ for every $n=1,\ldots,m-1$: by
  \eqref{eq:wn} it suffices to note that, for any $j=1,\ldots,n$ and
  $k_1+\cdots+k_n=j$,
  \[
  \norm[\Big]{D^jF(u) \Bigl( \bigl(v_1/1!\bigr)^{\otimes k_1},\ldots,%
    \bigl(v_n/n!\bigr)^{\otimes k_n} \Bigr)}_E \lesssim
  \norm[\big]{D^jF(u)}_{\cL_k(\sC^p;E)} \norm[\big]{v_1}^{k_1}_{\sC^p}
  \cdots \norm[\big]{v_n}^{k_n}_{\sC^p},
  \]
  where the right-hand side is finite because $u \in \sC^p$ and $F \in
  C^m(\sC^p;E)$ by assumption, and $v_1,\ldots,v_{m-1} \in \sC^p$ by
  Proposition~\ref{prop:ina}.
  The proof that $a_n \in E$ for all $n=m,\ldots,(m-1)m$, with norms
  bounded uniformly for $\varepsilon \in [0,\delta]$, is entirely
  similar, as it immediately follows by Lemmas~\ref{lm:ri} and
  \ref{lm:trep}.

  Finally, by multilinearity of $D^mF$, the integral on the right-hand
  side of \eqref{eq:fufi} can be written as $\sum_{n=m}^{m^2}
  b_n\varepsilon^n$, where $b_n$ depends on $\varepsilon$. By a
  reasoning entirely similar to the previous ones, in order to prove
  that $b_n \in E$ for all $n$ and that their $E$-norms are bounded as
  $\varepsilon \to 0$, we proceed as follows: $D^mF$ is continuous,
  hence bounded on a neighborhood $U$ of $u$. Since $u_\varepsilon \to
  u$ in $\sC^p$ as $\varepsilon \to 0$ by assumption, $u_\varepsilon
  \in U$ for $\varepsilon$ sufficiently small, hence also
  $tu_\varepsilon+(1-t)u \in U$, so that $D^mF(tu_\varepsilon+(1-t)u)$
  is bounded in $\cL_k(\sC^p;E)$ uniformly over $\varepsilon$ in a
  (right) neighborhood of zero and $t \in [0,1]$. Minkowski's
  inequality now implies that the $E$-norm of each $b_n$ can be
  estimated uniformly with respect to $\varepsilon$.
  Setting
  \[
  R_{m,\varepsilon} := \sum_{n=m}^{(m-1)m} a_n\varepsilon^n +
  \sum_{n=m}^{m^2} b_n\varepsilon^n,
  \]
  the proof is completed.
\end{proof}

We now consider the case where $F$ is defined only on $C([0,T];H)$.
\begin{thm}
  Let $E_0$ be a Banach space, $F\colon C([0,T];H) \to E_0$ of class
  $C^m$.  Assume that there exists $\beta \geq 0$ such that
  \[
  \norm[\big]{D^jF(x)}_{\cL_j(C([0,T];H);E)} \lesssim 1 +
  \norm[\big]{x}^\beta_{C([0,T];H)} \qquad \forall j \leq m
  \]
  and let $q > 0$ be defined by
  \[
  \frac{\beta+m}{p} = \frac{1}{q}.
  \]
  Then the conclusions of Theorem~\ref{thm:ae} hold with
  $E:=L^q(\Omega;E_0)$.
\end{thm}
\begin{proof}
  Taylor's theorem implies that \eqref{eq:fufi}, \eqref{eq:ue} and
  \eqref{eq:dudu} still hold, as an identities of $E_0$-valued random
  variables. As in the proof of the previous theorem, the integral on
  the right-hand side of \eqref{eq:fufi} can be written as the finite
  series $\sum_{n=m}^{m^2} b_n \varepsilon^n$, with each $b_n$
  possibly depending on $\varepsilon$. 
  We have to show that $w_n, a_n, b_n \in L^q(\Omega;E_0)$ for every
  $n$, and that the elements in $(a_n)$ and $(b_n)$ that depends on
  $\varepsilon$ remain bounded in $L^q(\Omega;E_0)$ as
  $\varepsilon \to 0$. To this purpose, denoting the norms of
  $C([0,T];H)$ and $\cL_j(C([0,T];H);E_0)$ by $\norm{\cdot}$ and
  $\norm{\cdot}_{\cL_j}$, respectively, for simplicity of notation,
  note that one has, for any $j \leq m$,
  \[
  \norm[\Big]{D^jF(u) \Bigl(%
  v_1^{\otimes k_1},\ldots,v_n^{\otimes k_n},\bar{R}_{m,\varepsilon}^{\otimes k_{n+1}}%
  \Bigr)}_E \leq \norm[\big]{D^jF(u)}_{\cL_j} \norm[\big]{v_1}^{k_1} \cdots
  \norm[\big]{v_n}^{k_n} \norm[\big]{\bar{R}_{m,\varepsilon}}^{k_{n+1}},
  \]
  where $\norm{D^jF(u)}_{\cL_j} \lesssim 1 + \norm{u}^\beta$ by
  assumption and $k_1+\cdots+k_{n+1}=j \leq m$, hence
  \[
  \frac{\beta}{p} + \frac{k_1}{p} + \cdots + \frac{k_{n+1}}{p} \leq
  \frac{\beta}{p} + \frac{m}{p} = \frac{1}{q}.
  \]
  Applying H\"older's inequality with the exponents implied by this
  inequality yields
  \begin{align*}
    &\norm[\Big]{D^jF(u) \Bigl(%
    v_1^{\otimes k_1},\ldots,v_n^{\otimes k_n},\bar{R}_{m,\varepsilon}^{\otimes k_{n+1}}%
    \Bigr)}_{L^q(\Omega;E)}\\
    &\hspace{3em} \lesssim \bigl( 1 + \norm{u}^\beta_{\sC^p} \bigr)
    \norm[\big]{v_1}_{\sC^p} \cdots
    \norm[\big]{v_n}_{\sC^p}
    \norm[\big]{\bar{R}_{m,\varepsilon}}_{\sC^p},
  \end{align*}
  where we have used the identity
  $\norm{z^\beta}_{L^{p/\beta}(\Omega)} =
  \norm{z}^\beta_{L^p(\Omega)}$, which holds for any positive random
  variable $z$.  Recalling that $\bar{R}_{m,\varepsilon}$ is bounded
  in $\sC^p$ uniformly over $\varepsilon \in [0,\delta]$, the claim
  about $(w_n)$ and $(a_n)$ is proved. In order to show that $(b_n)$
  enjoys the same properties of $(a_n)$, it is immediately seen that
  it suffices to bound the norm of $D^mF(tu_\varepsilon + (1-t)u)$ in
  $L^{q/\beta}(\Omega;\cL_m(C(0,T;H);E_0))$, uniformly with respect to
  $\varepsilon$ in a right neighborhood of zero. But
  \[
  \norm[\big]{D^mF(tu_\varepsilon + (1-t)u)}_{\cL_m} \lesssim 1 +
  \norm[\big]{u+t(u_\varepsilon - u)}^\beta
  \]
  implies   
  \begin{align*}
    \norm[\big]{D^mF(tu_\varepsilon + (1-t)u)}_{L^{p/\beta}(\Omega;\cL_m)}
    &\lesssim 1 + \norm[\big]{u+t(u_\varepsilon - u)}_{\sC^p}^\beta\\
    &\lesssim 1 + \norm[\big]{u}^\beta_{\sC^p}
    + \norm[\big]{u_\varepsilon - u}^\beta_{\sC^p},
  \end{align*}
  where the norm in $\sC^p$ of $u_\varepsilon - u$ tends to zero as
  $\varepsilon \to 0$. The proof is thus completed.
\end{proof}
\begin{rmk}
  One could have also approached the problem in a more abstract way,
  establishing conditions implying that the function $F$ can be
  ``lifted'' to a function of class $C^m$ from $\sC^p$ to
  $E=L^q(\Omega;E_0)$, and then applying the Theorem~\ref{thm:ae}.  We
  have preferred the above more direct way because it could also be
  applicable, \emph{mutatis mutandis}, in situations where $F$ admits
  a series representation not necessarily of Taylor's type.
\end{rmk}


\section{Singular perturbations of a transport equation and the
  Musiela SPDE}
\label{sec:Mus}
\subsection{A transport equation}
\label{ssec:t}
Let $w$ be a fixed strictly positive real number and set, for
notational convenience, $L^2_w:=L^2(\erre,e^{wx}dx)$. Let $H$ be the
Hilbert space of absolutely continuous functions
$f \in L^1_{\mathrm{loc}}(\erre)$ such that $f' \in L^2_w$, equipped
with the scalar product
\[
\ip{f}{g} := \lim_{x \to +\infty} f(x)g(x) + \ip[\big]{f'}{g'}_{L^2_w}.
\]
The definition is well posed because
$f(+\infty):=\lim_{x \to +\infty} f(x)$ exists and is finite for every
$f \in H$. In fact, for any $a \in \erre$ such that $f(a)$ is
finite, one has
\begin{equation}
\label{eq:fixa}
  \abs{f(x) - f(a)} \leq \int_a^x \abs{f'(y)}\,dy
  \leq \biggl( \int_a^x \abs{f'(y)}^2 e^{wy}\,dy \biggr)^{1/2}
  \biggl( \int_a^\infty e^{-wy}\,dy \biggr)^{1/2}.
\end{equation}
We shall denote the norm in $H$ induced by the above scalar product as
$\norm{\cdot}$. The following simple consequence of the definition of
$H$ will be repeatedly used below: if $f \in H$, then
\[
\lim_{x \to \pm \infty} \abs{f'(x)}^2 e^{wx} = 0,
\]
in particular $\lim_{x\to+\infty} f'(x)=0$.

Let $S_A$ be the strongly continuous semigroup on $H$ defined by
$[S_A(t)f](x):=f(t+x)$. The elementary identity
\[
  \int_\erre \abs{f'(t+x)}^2 e^{wx}\,dx =
  e^{-wt} \int_\erre \abs{f'(x)}^2 e^{wx}\,dx
\]
implies that $S_A$ is a contraction semigroup. Its generator is the
maximal dissipative operator $A\colon f \mapsto f'$ defined on
\[
\dom(A) = \bigl\{ f \in H:\, f' \in H \bigr\}.
\]

One has the following formula of integration by parts.
\begin{lemma}
  If $f$, $g \in \dom(A)$, then
  \[
    \ip{Af}{g} = -\ip{f}{Ag} - w\ip[\big]{f'}{g'}_{L^2_w}.
  \]
  In particular,
  \[
    \ip{Af}{f} = -\frac{w}{2} \norm[\big]{f'}^2_{L^2_w}.
  \]
\end{lemma}
\begin{proof}
  By definition, one has
  \[
  \ip{Af}{g} = f'(+\infty)g(+\infty) + \int_\erre f''(x)g'(x)e^{wx}\,dx,
  \]
  where $f'(+\infty)=0$ and $g(+\infty)$ is finite, hence, integrating
  by parts,
  \begin{align*}
    \ip{Af}{g}
    &= \int_\erre f''(x)g'(x) e^{wx}\,dx\\
    &= \lim_{x\to+\infty} f'(x)g'(x)e^{wx} - \lim_{x\to-\infty} f'(x)g'(x)e^{wx}\\ 
    &\quad - \int_\erre f'(x)g''(x) e^{wx}\,dx
      - w \int_\erre f'(x)g'(x)e^{wx}\,dx\\
    &= - \ip{f}{Ag} - w\ip[\big]{f'}{g'}_{L^2_w},
  \end{align*}
  where the two limits are equal to zero by the elementary estimate
  $f'(x)g'(x)e^{wx} \leq \abs{f'(x)}^2e^{wx} + \abs{g'(x)}^2e^{wx}$.
\end{proof}

Let us now consider the operator $A^2$, defined on its natural domain
$\dom(A^2)$ of elements $f \in \dom(A)$ such that $Af \in \dom(A)$.
\begin{lemma}
  The operator $A^2$ is quasi-dissipative. More precisely, $A^2-(w^2/2)I$
  is dissipative.
\end{lemma}
\begin{proof}
  Let $f \in \dom(A^2)$, and substitute $g=Af$ in the integration by
  parts formula of the previous lemma. We get
  \[
    \norm{Af}^2 = -\ip{A^2f}{f} - w\ip[\big]{f'}{f''}_{L^2_w} =
    -\ip{A^2f}{f} - w\ip{Af}{f}
  \]
  i.e.
  \[
    \ip{A^2f}{f} + w\ip{Af}{f} = - \norm{Af}^2,
  \]
  hence also
  \[
    \ip{A^2f}{f} - \frac{w^2}{2} \norm[\big]{f'}^2_{L^2_w} =
    - \norm{Af}^2
  \]
  and
  \[
  \ip{A^2f}{f} - \frac{w^2}{2} \norm{f}^2 = - \norm{Af}^2
  - \frac{w^2}{2} \abs{f(+\infty)}^2 \leq 0.
  \qedhere
  \]
\end{proof}

\begin{prop}
  \label{prop:gai}
  The operator $G:=A^2 - (w^2/2)I$ is maximal dissipative in $H$.
\end{prop}
\begin{proof}
  The dissipativity of $G$ has already been proved. Moreover, $A^2$ is
  closed, as is every integer positive power of the generator of a
  strongly continuous semigroup (see, e.g.,
  \cite[Proposition~1.1.6]{BuBe}). Hence we only have to show that
  there exist $\lambda>0$ such that the image of $\lambda I-G$ is
  $H$. To this purpose, let $f \in H$ and consider the equation
  $\lambda y - y'' = f$, which yields $\lambda y'-y'''=f'$. Defining
  (formally, for the moment) $z$ through $y'(x)=z(x)e^{-wx/2}$, one has
  \begin{equation}
    \label{eq:uai3}
    y'''(x) e^{wx/2} =  z''(x) - wz'(x) + \frac{w^2}{4} z(x),
  \end{equation}
  hence
  \[
    \lambda y'(x)e^{wx/2}  - y'''(x) e^{wx/2}
    = \bigl( \lambda - w^2/4 \bigr) z(x) + wz'(x) - z''(x).
  \]
  We are thus led to consider the equation
  \[
    \bigl( \lambda - w^2/4 \bigr) z + wz' - z'' = \widetilde{f},
    \qquad \widetilde{f}(x) := f'(x) e^{wx/2}.
  \]
  Let us introduce the bounded bilinear form $a$ on $H^1:=H^1(\erre)$
  defined as
  \[
    a(\varphi,\psi) := \bigl( \lambda - w^2/4 \bigr) \int_\erre \varphi\psi
    + w \int_\erre \varphi' \psi + \int_\erre \varphi' \psi'.
  \]
  It is immediately seen that, for any $\lambda>w^2/4$, the bilinear
  form $a$ is coercive on $H^1$, hence the Lax-Milgram theorem yields
  the existence and uniqueness of a (weak) solution $z \in
  H^1$. Moreover, the equation satisfied by $z$ implies that, in fact,
  $z \in H^2$. This immediately yields the existence of a solution $y$
  to $\lambda y - y''=f$. Moreover, by definition of $z$ it is
  immediate that $y \in H$, the identity
  $y''(x)e^{wx/2}=z'(x)-wz(x)/2$ implies that $y'' \in L^2_w$, and
  \eqref{eq:uai3} implies that $y''' \in L^2_w$, i.e.
  $y \in \dom(A^2)$, thus completing the proof.
\end{proof}

Since $A$ is maximal dissipative, the transport equation on $H$
\[
du = Au\,dt + \alpha(u)\,dt + B(u)\,dW, \qquad u(0)=u_0,
\]
with $\alpha$ and $B$ satisfying the measurability and Lipschitz
continuity assumptions of \S\ref{sec:3} and $u_0 \in L^p(\Omega;H)$,
$p>0$, admits a unique mild solution $u \in \sC^p$ (see, e.g., as
already mentioned, \cite[Chapter~7]{DZ92} and \cite{cm:SIMA18}). Under
the same assumptions on $\alpha$, $B$, and $u_0$, the singularly
perturbed equation
\[
  du_\varepsilon = (A + \varepsilon G)u_\varepsilon\,dt
  + \alpha(u_\varepsilon)\,dt + B(u_\varepsilon)\,dW,
  \qquad u_\varepsilon(0)=u_0,
\]
admits a unique mild solution $u_\varepsilon \in \sC^p$, which
converges to $u$ in $\sC^p$ as $\varepsilon \to 0$.
Furthermore, if the coefficients $\alpha$ and $B$ do not depend on $u$
and there exists an integer number $m \geq 1$ such that
\begin{gather*}
  u_0 \in L^p(\Omega;\dom(A^{2m})), \quad %
  \alpha \in L^p(\Omega;L^1(0,T;\dom(A^{2m}))),\\
  B \in L^p(\Omega;L^2(0,T;\cL^2(U;\dom(A^{2m})))),
\end{gather*}
then we can construct a series expansion of $u_\varepsilon$ around $u$
of order $m-1$, applying the results of \S\ref{sec:exp}.

\subsection{Parabolic approximation of Musiela's SPDE}
Let $u(t,x)$, $t,x \geq 0$, denote the instantaneous forward rate at
time $t$ with maturity $t+x$. Musiela observed that the equation for
forward rates in the Heath-Jarrow-Morton model can be written as (the
mild form of)
\begin{equation}
  \label{eq:Mus}
  du(t,x) = \partial_x u(t,x)\,dt + \alpha_0(t,x)\,dt +
  \sum_{k=1}^\infty \sigma_k(t,x)dw^k(t),
\end{equation}
where $(w^k)_{k \in \mathbb{N}}$ is a sequence of standard real Wiener
processes, the volatilities $\sigma_k$ are possibly random, and
$\alpha_0$ is uniquely determined by $(\sigma_k)$ if the reference
probability measure is such that implied discounted bond prices are
local martingales. In particular, in this case it must necessarily
hold
\[
  \alpha_0(t,x) = \sum_{k=1}^\infty \sigma_k(t,x) \int_0^x
  \sigma_k(t,y)\,dy.
\]
For details on the financial background we refer to, e.g.,
\cite{CarTeh,filipo,HJM,musiela,MusRut}.
There is a large literature on the well-posedness of \eqref{eq:Mus} in
the mild sense, also in the (more interesting) case where
$(\sigma_k)$, hence $\alpha_0$, depend explicitly on the unknown $u$,
with different choices of state space as well as with more general
noise (see, e.g.,
\cite{BaZa:MusLe,BrzKok,filipo,Kusu:Mus,cm:MF10,vargiolu},
\cite[{\S}20.3]{PZ-libro}). Here we limit ourselves to the case where
$(\sigma_k)$ are possibly random, but do not depend explicitly on $u$,
and use as state space $H(\erre_+)$, which we define as the space of
locally integrable functions on $\erre_+$ such that
$f' \in L^2(\erre_+,e^{wx}\,dx)$, endowed with the inner product
\[
\ip{f}{g} = f(+\infty)g(+\infty) + \int_0^{+\infty} f'(x)g'(x) e^{wx}\,dx.
\]
This choice of state space, introduced in \cite{filipo} (cf. also
\cite{tehranchi}), to which we refer for further details, is standard and
enjoys many good properties from the point of view of financial
modeling. For instance, forward curves are continuous and can be
``flat'' at infinity without decaying to zero.

In order to give a precise notion of solution to \eqref{eq:Mus}, we
recall that the semigroup of left translation on $H(\erre_+)$ is
strongly continuous and contractive, and that its generator is
$A_0\colon \phi \mapsto \phi'$ on the domain
$\dom(A_0) = \{\phi:\, \phi' \in H(\erre_+)\}$ (see \cite{filipo}).
Moreover, let us assume that there exists $p>0$ such that
\begin{equation}
  \label{eq:sigma}
  \E \biggl( \sum_{k=1}^\infty \int_0^T
  \norm[\big]{\sigma_k(t,\cdot)}^2_{H(\erre_+)} dt \biggr)^{p} <
  \infty, \qquad \sigma_k(t,+\infty)=0 \quad \forall k \in \mathbb{N},
\end{equation}
so that the random time-dependent linear map
\begin{align*}
  B_0(\omega,t) \colon \ell^2 &\longrightarrow H(\erre_+)\\
  (h_k) &\longmapsto \sum_{k=1}^\infty \sigma_k(\omega,t,\cdot) h_k
\end{align*}
belongs to
$L^{2p}(\Omega;L^2(0,T;\cL^2(\ell^2,H(\erre_+))))$. Setting, for any
$\phi \in L^1_{\mathrm{loc}}(\erre_+)$,
\[
[I\phi](x) := \int_0^x \phi(y)\,dy, \qquad x \geq 0,
\]
one has the basic estimate
\[
  \norm[\big]{\phi\,I\phi}_{H(\erre_+)} \lesssim
  \norm[\big]{\phi}^2_{H(\erre_+)}
\]
for every $\phi \in H(\erre_+)$ such that $\phi(+\infty)=0$
(cf.~\cite{filipo}, or see Lemma~\ref{lm:fif} below for a proof in a
more general setting). This implies that the assumption on
$(\sigma_k)$ yields $\alpha_0 \in L^p(\Omega;L^1(0,T;H(\erre_+)))$. We
then have the following well-posedness result for \eqref{eq:Mus},
written in its abstract form as
\begin{equation}
  \label{eq:musa}
  du + A_0u\,dt = \alpha_0\,dt + B_0\,dW, \qquad u(0)=u_0,
\end{equation}
where $W$ is a cylindrical Wiener process on $U:=\ell^2$.
\begin{prop}
  Let $p>0$. Assume that $u_0 \in L^p(\Omega,\cF_0;H(\erre_+))$ and
  \eqref{eq:sigma} is satisfied. Then \eqref{eq:musa} has a unique
  mild solution
  $u \in \sC^p(H(\erre_+)) := L^p(\Omega;C([0,T];H(\erre_+)))$, which
  depends continuously on the initial datum $u_0$.
\end{prop}

Musiela's equation \eqref{eq:musa} is closely related to the transport
equation studied in \S\ref{ssec:t} above, the main difference being
the state space. In the following we shall denote the state space of
the transport equation by $H(\erre)$.

As mentioned in the introduction, it has been suggested
(see~\cite{Cont:TSIR} and references therein) that second-order
parabolic SPDEs, with respect to the physical probability measure,
capture several empirical features of observed forward rates. It seems
reasonable to assume that such SPDEs would retain their parabolic
character even after changing the reference probability measure to one
with respect to which discounted bond prices are (local) martingales,
thus excluding arbitrage. It is then natural to consider singular
perturbations of the Musiela equation on $H(\erre_+)$ adding a
singular term $\varepsilon G$ to the drift $A_0$ in \eqref{eq:musa},
with $G=A_0^2$, which is, roughly speaking, a second derivative in the
time to maturity. On the other hand, if forward rates satisfy the
general assumptions of the Heath-Jarrow-Morton model, the HJM drift
condition is sufficient and necessary for discounted bond prices to be
local martingales. Therefore singular perturbations of the Musiela
SPDE introduce arbitrage, in the sense that the implied discounted
bond prices may not be local martingales. It is hence interesting to
``quantify'' and control the amount of arbitrage introduced by a
parabolic perturbation of the Musiela SPDE \eqref{eq:musa}.
The arguments used in \S\ref{ssec:t} for the transport equation,
however, give rise to major problems, mainly because boundary terms
(at zero) appear that seem impossible to control. To circumvent such
issues, we ``embed'' the abstract Musiela equation \eqref{eq:musa}
into a transport equation on $H(\erre)$ of the type considered in
\S\ref{ssec:t}, we perturb the equation thus obtained, get asymptotic
expansions, and finally ``translate'' back the results, in a suitable
sense, to the Musiela equation.

We need some technical results first.
Let $H_0(\erre_+)$ be the Hilbert space of functions in $H(\erre_+)$
that are zero at infinity. The following embeddings and estimates are
rather straightforward (see \cite{filipo} for a proof) and will be
repeatedly used below:
\begin{itemize}
\item[(i)] $H(\erre_+) \embed C_b(\erre_+)$;
\item[(ii)] $H_0(\erre_+) \embed L^1(\erre_+)$;
\item[(iii)] $H_0(\erre_+) \embed L^4_w(\erre_+)=L^4(\erre_+,e^{wx}\,dx)$.
\end{itemize}

Let $H^m_w(\erre_+)$ be the set of functions in
$L^1_{\mathrm{loc}}(\erre_+)$ that belong to $L^2_w(\erre_+)$ together
with all their derivatives up to order $m$, endowed with the norm
defined by
\[
\norm[\big]{f}_{H^m_w(\erre_+)} = 
\sum_{k=0}^m \norm[\big]{f^{(k)}}_{L^2_w(\erre_+)}.
\]

\begin{lemma}
\label{lm:abc}
  Let $f \in L^1_{\mathrm{loc}}(\erre_+)$ and $m$ a positive
  integer. The following assertions are equivalent: \emph{(a)} $f \in
  \dom(A_0^m)$; \emph{(b)} $f' \in H^m_w(\erre_+)$; \emph{(c)} $x
  \mapsto f'(x)e^{wx/2} \in H^m(\erre_+)$. Moreover, for any $f \in
  \dom(A_0^m)$ with $f(+\infty)=0$,
  \[
  \norm[\big]{f}_{\dom(A_0^m)} = \norm[\big]{f'}_{H^m_w(\erre_+)} \eqsim
  \norm[\big]{f' e^{w\cdot/2}}_{H^m(\erre_+)},
  \]
  where the implicit constant depends only on $m$ and $w$.
\end{lemma}
\begin{proof}
  The equivalence of (a) and (b) is immediate by the definition of
  $A_0$ and by an inequality completely analogous to
  \eqref{eq:fixa}. In particular, if $f(+\infty)=0$, the identity
  ${\norm{f}}_{\dom(A_0^m)} = {\norm{f'}}_{H^m_w(\erre_+)}$ is
  a tautology. The other assertions follow by the identity
  \[
  \bigl( f' e^{w\cdot/2} \bigr)^{(n)} = \sum_{j=0}^n \binom{n}{j} (w/2)^{n-j}
  f^{(j+1)} e^{w\cdot/2} \qquad \forall n \in \{1,\ldots,m\}.
  \qedhere
  \]
\end{proof}

\begin{lemma}
  \label{lm:fif}
  Let $m \geq 1$ be an integer. If $f \in \dom(A_0^m)$ with
  $f(+\infty)=0$, then
  \[
  \norm[\big]{f\,If}_{\dom(A_0^m)} \lesssim
  \norm[\big]{f}^2_{\dom(A_0^m)},
  \]
  where the implicit constant depends only on $m$ and $w$.
\end{lemma}
\begin{proof}
  Let $f \in \dom(A_0^m)$ with $f(+\infty)=0$. In view of the previous
  lemma, we will bound the $H^m_w(\erre_+)$ norm of $(f\,If)' = f'\,If
  + f^2$ in terms of the $H^m_w(\erre_+)$ norm of $f'$.
  One has, omitting the indication of $\erre_+$ in the notation,
  \[
  \norm[\big]{f'\,If}_{L^2_w} \leq 
  \norm[\big]{f'}_{L^2_w} \norm[\big]{If}_{L^\infty} \leq
  \norm[\big]{f'}_{L^2_w} \norm[\big]{f}_{L^1} \lesssim
  \norm[\big]{f'}_{L^2_w} \norm[\big]{f'}_{L^2_w}
  \]
  and
  \[
  \norm[\big]{f^2}_{L^2_w} = \norm[\big]{f}^2_{L^4_w} \lesssim
  \norm[\big]{f'}^2_{L^2_w}.
  \]
  Let $1 \leq n \leq m$ be an integer. One has
  \[
  \bigl( f'\,If \bigr)^{(n)} = \sum_{j=0}^n \binom{n}{j}
  f^{(j+1)} (If)^{(n-j)}
  \]
  and
  \[
  \norm[\big]{f^{(j+1)} (If)^{(n-j)}}_{L^2_w} \leq
  \norm[\big]{f^{(j+1)}}_{L^2_w}
  \norm[\big]{(If)^{(n-j)}}_{L^\infty},
  \]
  where, if $j=n$,
  \[
  \norm[\big]{(If)^{(n-j)}}_{L^\infty} = \norm[\big]{If}_{L^\infty}  
  \leq \norm[\big]{f}_{L^1} \lesssim \norm[\big]{f'}_{L^2_w},
  \]
  while, if $j \leq n-1$,
  \[
  \norm[\big]{(If)^{(n-j)}}_{L^\infty} =
  \norm[\big]{f^{(n+1-j)}}_{L^\infty} \lesssim
  \norm[\big]{f^{(n-j)}}_{L^2_w}.
  \]
  Similarly,
  \[
  {(f^2)}^{(n)} = \sum_{j=0}^n \binom{n}{j} f^{(j)} f^{(n-j)}
  = 2f\,f^{(n)} + \sum_{j=1}^{n-1} \binom{n}{j} f^{(j)} f^{(n-j)}
  \]
  where
  \[
  \norm[\big]{f\,f^{(n)}}_{L^2_w} \leq
  \norm[\big]{f}_{L^\infty} \norm[\big]{f^{(n)}}_{L^2_w}
  \lesssim   \norm[\big]{f'}_{L^2_w} \norm[\big]{f^{(n)}}_{L^2_w}
  \]
  and, if $1 \leq j \leq n-1$,
  \[
  \norm[\big]{f^{(j)} f^{(n-j)}}_{L^2_w} \leq
  \norm[\big]{f^{(j)}}_{L^\infty} \norm[\big]{f^{(n-j)}}_{L^2_w}
  \lesssim \norm[\big]{f^{(j+1)}}_{L^2_w} \norm[\big]{f^{(n-j)}}_{L^2_w}.
  \]
  The claim is then an immediate consequence of these estimates.
\end{proof}

\begin{prop}
  \label{prop:boa}
  Let $p>0$ and $m$ a positive integer. If 
  \[
  \E \biggl( \sum_{k=1}^\infty \int_0^T
  \norm[\big]{\sigma_k(t,\cdot)}^2_{\dom(A_0^{2m})}\,dt \biggr)^p < \infty
  \]
  or, equivalently, $B_0 \in
  L^{2p}(\Omega;L^2(0,T;\cL^2(\ell^2;\dom(A_0^{2m}))))$, then
  $\alpha_0 \in L^p(\Omega;L^1(0,T;\dom(A_0^{2m})))$.
\end{prop}
\begin{proof}
  One has
  \[
  \norm[\big]{B_0(t)}^2_{\cL^2(\ell^2;\dom(A_0^{2m}))} = 
  \sum_{k=1}^\infty \norm[\big]{\sigma_k(t,\cdot)}^2_{\dom(A_0^{2m})}
  \]
  and, by the previous lemma,
  \begin{align*}
    \norm[\big]{\alpha_0(t)}_{\dom(A_0^{2m})} 
    &= \norm[\bigg]{\sum_{k=1}^\infty %
      \sigma_k(t,\cdot) \, I\sigma_k(t,\cdot)}_{\dom(A_0^{2m})}\\
    &\leq \sum_{k=1}^\infty 
    \norm[\big]{\sigma_k(t,\cdot) \, I\sigma_k(t,\cdot)}_{\dom(A_0^{2m})}\\
    &\lesssim \sum_{k=1}^\infty
    \norm[\big]{\sigma_k(t,\cdot)}^2_{\dom(A_0^{2m})} =
    \norm[\big]{B(t)}^2_{\cL^2(\ell^2;\dom(A_0^{2m}))},
  \end{align*}
  hence
  \begin{align*}
    \norm[\big]{\alpha_0}_{L^p(\Omega;L^1(0,T;\dom(A_0^{2m})))}
    &\lesssim
    \norm[\Big]{\norm[\big]{B_0}^2_{\cL^2(\ell^2;\dom(A_0^{2m}))}}_{L^p(\Omega;L^1(0,T))}\\
    &= \norm[\big]{B_0}_{L^{2p}(\Omega;L^2(0,T;\cL^2(\ell^2;\dom(A_0^{2m}))))}.
    \qedhere
  \end{align*}
\end{proof}

\begin{lemma}
  There exists a linear continuous extension operator
  $L\colon \dom(A_0^{2m}) \to \dom(A^{2m})$.
\end{lemma}
\begin{proof}
  By an extension result due to Stein (see \cite[p.~181]{Stein}),
  there exists a linear continuous extension operator $L_0 \colon
  H^{2m}(\erre_+) \to H^{2m}(\erre)$.  Since a locally integrable
  function $f$ belongs to $\dom(A_0^{2m})$ if and only if $x \mapsto
  f'(x)e^{wx/2} \in H^{2m}(\erre_+)$ by Lemma~\ref{lm:abc} and $f \in
  H(\erre_+)$ implies that $f(+\infty)$ is finite, the map
  \[
    L\colon f \longmapsto \Bigl[ x \mapsto f(+\infty) -
    \int_x^{+\infty} e^{-wy/2} L_0\bigl( f' e^{w\cdot/2}
    \bigr)(y)\,dy \Bigr]
  \]
  is well defined on $\dom(A_0^{2m})$. Moreover, $L_0\bigl( f'
  e^{w\cdot/2} \bigr) \in H^{2m}(\erre)$, hence $y \mapsto e^{-wy/2}
  L_0\bigl( f' e^{w\cdot/2} \bigr)(y) \in L^1(x,+\infty)$ by Cauchy's
  inequality for all $x \in \erre$, so that $Lf(x)$ is finite for
  every $x \in \erre$ and $Lf(+\infty):=\lim_{x\to+\infty} Lf(x) =
  f(+\infty)$. Moreover,
  \[
    x \mapsto e^{wx/2} (Lf)'(x) = L_0\bigl( f' e^{w\cdot/2}
    \bigr) \in H^{2m}(\erre),
  \]
  hence $Lf \in \dom(A^{2m})$ by an argument completely analogous to
  the proof of Lemma~\ref{lm:abc}. Finally,
  \begin{align*}
    \norm[\big]{Lf}_{\dom(A^{2m})}
    &\lesssim \abs{f(+\infty)}
      + \norm[\big]{(Lf)' e^{w\cdot/2}}_{H^{2m}(\erre)}\\
    &= \abs{f(+\infty)}
      + \norm[\big]{L_0(f' e^{w\cdot/2})}_{H^{2m}(\erre)}\\
    &\lesssim \abs{f(+\infty)} 
      + \norm[\big]{f' e^{w\cdot/2}}_{H^{2m}(\erre_+)}\\
    &\lesssim \norm[\big]{f}_{\dom(A_0^{2m})}.
    \qedhere
  \end{align*}
\end{proof}

\medskip

Let $L$ be the extension operator just introduced and set $v_0 :=
Lu_0$, $\alpha := L\alpha_0$, and $B := LB_0$, with $\alpha_0$ and
$B_0$ as in Proposition~\ref{prop:boa}, so that
\begin{gather*}
  v_0 \in L^p(\Omega,\cF_0;\dom(A^{2m})), \qquad
  \alpha \in L^p(\Omega;L^1(0,T;\dom(A^{2m}))),\\
  B \in L^p(\Omega;L^2(0,T;\cL^2(\ell^2;\dom(A^{2m})))),
\end{gather*}
and consider the following stochastic equation in $H(\erre)$:
\begin{equation}
  \label{eq:v}
  dv = Av\,dt + \alpha\,dt + B\,dW, \qquad v(0)=v_0, \quad t \geq 0,
\end{equation}
where $A$ is the generator of the semigroup of translation on
$H(\erre)$ and $W$ is a cylindrical Wiener process on $\ell^2$. By the
discussion at the beginning of this section, this equation admits a
unique mild solution $v \in \sC^p(\dom(A^{2m}))$, which is thus also a
strong solution, i.e. such that
\[
  v(t) = v_0 + \int_0^t Av(s)\,ds + \int_0^t \alpha(s)\,ds
  + \int_0^t B(s)\,dW(s),
\]
where the equality is in the sense of indistinguishable
$H(\erre)$-valued (hence also $C(\erre)$-valued) processes. In a more
explicit form, one has
\[
  v(t,x) = v_0(x) + \int_0^t \partial_x v(s,x)\,ds + \int_0^t \alpha(s,x)\,ds
  + \sum_{k=1}^\infty \int_0^t \sigma_k(s,x)\,dw^k(s)
\]
for every $x \in \erre$, in particular for every $x \in
\erre_+$. Since the restriction of $v_0$, $\alpha$ and $B$ to
$\erre_+$ are equal to $u_0$, $\alpha_0$ and $B_0$, respectively, the
restriction of $v$ to $\erre_+$ must coincide with the unique strong
solution in $H(\erre_+)$ to the Musiela equation \eqref{eq:musa}.

Moreover, the equation in $H(\erre)$
\begin{equation}
  \label{eq:ve}
  dv_\varepsilon = (A + \varepsilon A^2)v_\varepsilon\,dt + \alpha\,dt +
  B\,dW, \qquad v_\varepsilon(0)=v_0,
\end{equation}
also admits a unique mild solution $v_\varepsilon \in
\sC^p(\dom(A^{2m}))$, which converges to $v$ in $\sC^p(\dom(A^{2m}))$
as $\varepsilon \to 0$, and satisfies an identity of the type
\[
  v_\varepsilon - v = \sum_{k=1}^{m-1} \frac{v_k}{k!} \varepsilon^k +
  R_{m,\varepsilon}
\]
in $H(\erre)$, in particular in $C(\erre)$, where
$v_1,\ldots,v_{m-1},R_{m,\varepsilon} \in \sC^p(H(\erre))$ and
$R_{m,\varepsilon}/\varepsilon^{m-1}$ tends to zero as $\varepsilon
\to 0$. Taking the $H(\erre_+)$ norm on both sides yields
\[
  \norm[\big]{v_\varepsilon - v}_{H(\erre_+)}
  \leq \sum_{k=1}^{m-1} \frac{1}{k!} \varepsilon^k \norm[\big]{v_k}_{H(\erre_+)}
    + \norm[\big]{R_{m,\varepsilon}}_{H(\erre_+)},
\]
where all $H(\erre_+)$ norms involved are finite because they are
dominated by the corresponding ones in $H(\erre)$, that are finite.
We have thus proved the following.
\begin{thm}
  Let $p>0$ and $m \geq 1$ be a positive integer such that
  \[
  \E \biggl( \sum_{k=1}^\infty \int_0^T
  \norm[\big]{\sigma_k(t,\cdot)}^2_{\dom(A_0^{2m})}\,dt \biggr)^p < \infty.
  \]
  Then equation \eqref{eq:v} has a unique strong solution $v$ in
  $\sC^p(H(\erre))$ and its restriction to $H(\erre_+)$ coincides with
  the unique strong solution $u$ in $\sC^p(H(\erre_+))$ to the Musiela
  equation \eqref{eq:musa}. Moreover, the restriction to $H(\erre_+)$
  of the mild solution $v_\varepsilon$ to the perturbed extended
  Musiela equation \eqref{eq:ve} converges to $v$ in
  $\sC^p(\dom(A^{2m}))$ and the estimate
  \[
  \norm[\big]{v_\varepsilon - u}_{\sC^p(H(\erre_+))}
  \leq \sum_{k=1}^{m-1} \frac{1}{k!} \varepsilon^k \norm[\big]{v_k}_{\sC^p(H(\erre_+))}
    + \norm[\big]{R_{m,\varepsilon}}_{\sC^p(H(\erre_+))}
  \]
  holds, with $\lim_{\varepsilon \to 0}
  R_{m,\varepsilon}/\varepsilon^{m-1}=0$ in $\sC^p(H(\erre_+))$.
\end{thm}

\medskip

We shall now consider bond prices and their approximation in diffusive
correction of Musiela's equation. The solutions $v$ and
$v_\varepsilon$ to the equations \eqref{eq:v} and \eqref{eq:ve} have
paths in $H(\erre)$, hence their restrictions $x \mapsto v(t,x)$ and
$x \mapsto v_\varepsilon(t,x)$, $x \in \erre_+$, belong to
$H(\erre_+)$ for every $t \in [0,T]$ and $u(t,x)=v(t,x)$ for every
$(t,x) \in [0,T] \times \erre_+$.
The price of a zero-coupon bond with face value equal to one at time
$t \geq 0$ with time to maturity $x \geq 0$ is given by
\[
\widehat{B}(t,x) = \exp \biggl( -\int_t^{t+x} v(t,t+y)\,dy \biggr)
= \exp \biggl( -\int_0^x v(t,y)\,dy \biggr),
\]
and the value at time $t$ of the money market account is given by
\[
\beta(t) = \exp\biggl( \int_0^t v(s,0)\,ds \biggr),
\]
hence the corresponding discounted price of the zero-coupon bond is
\[
  B(t,x) := \frac{\widehat{B}(t,x)}{\beta(t)}
  = \exp \biggl( -\int_0^x v(t,y)\,dy - \int_0^t v(s,0)\,ds \biggr)
\]
Let us define the discounted price of the (fictitious) zero coupon
bond associated to $v_\varepsilon$ as
\[
  B_\varepsilon(t,x) = \exp \biggl( -\int_0^x v_\varepsilon(t,y)\,dy
  -\int_0^t v_\varepsilon(s,0)\,ds \biggr).
\]
For fixed $t \in [0,T]$ and $x \geq 0$, let us define the linear map
\begin{align*}
F_{t,x} \colon C([0,T] \times \erre) &\longrightarrow \erre\\
f &\longmapsto \int_0^x f(t,y)\,dy + \int_0^t f(s,0)\,ds
\end{align*}
so that $B(t,x) = \exp\bigl( -F_{t,x}u \bigr)$ and
$B_\varepsilon(t,x)=\exp\bigl( -F_{t,x}u_\varepsilon \bigr)$.
\begin{lemma}
  Let $(t,x) \in [0,T] \times \erre_+$. The linear map $F_{t,x}$ is
  continuous
  \begin{itemize}
  \item[\emph{(i)}] from $C([0,T];H(\erre_+))$ to $\erre$, hence also
    from $C([0,T];H(\erre))$ to $\erre$, and
  \item[(ii)] from $\sC^p(H(\erre_+))$ to $L^p(\Omega;\erre)$, hence
    also from $\sC^p(H(\erre))$ to $L^p(\Omega;\erre)$, for every
    $p>0$.
  \end{itemize}
\end{lemma}
\begin{proof}
  For any $f \in C([0,T];H(\erre))$ one has
  \begin{align*}
    \abs[\bigg]{\int_0^x f(t,y)\,dy}
    &\leq \int_0^x \abs{f(t,y)-f(t,+\infty)}\,dy + \abs{f(t,+\infty)}x\\
    &\lesssim (1+x) \, {\norm{f(t)}}_{H(\erre_+)}
      \leq (1+x) \, {\norm{f}}_{C([0,T];H(\erre_+))}
  \end{align*}
  and
  \[
    \abs[\bigg]{\int_0^t f(s,0)\,ds} \leq \int_0^t
    {\norm{f(s)}}_{L^\infty(\erre_+)}\,ds
    \lesssim T {\norm{f}}_{C([0,T];H(\erre_+))},
  \]
  thus proving (i). Raising both sides of both inequalities to the
  power $p$ and taking expectations proves (ii).
\end{proof}

More generally, it is easy to show that the linear map $F$ defined as
\begin{align*}
  F \colon C([0,T] \times \erre)
  &\longrightarrow C([0,T] \times \erre)\\
  f &\longmapsto \Bigl[ (t,x) \mapsto \int_0^x f(t,y)\,dy
      + \int_0^t f(s,0)\,ds \Bigr]
\end{align*}
is continuous from $C([0,T];H(\erre_+))$ to $C([0,T] \times \erre)$,
endowed with the topology of uniform convergence on compact sets, as
well as from $\sC^p(H(\erre_+))$ to $L^p(\Omega;C([0,T] \times I))$
for every compact $I \subset \erre_+$.

The operators $F_{t,x}$ and $F$, being linear and continuous, are
automatically of class $C^\infty$ (assuming $p \geq 1$ when the norms
in the domain and codomain depend on such a parameter), with
$F'(z)=F(z)$ for every $z$ in the domain of $F$, and higher-order
derivatives equal to zero (and completely analogously for $F_{t,x}$).

\medskip

Given a series expansion of $v_\varepsilon$ around $v$ of the type
\[
v_\varepsilon - v = \sum_{k=1}^{m-1} \frac{1}{k!}
v_k \, \varepsilon^k + R_{m,\varepsilon},
\]
which can be considered as an identity in $\sC^p(H(\erre))$, as well
as in $\sC^p(H(\erre_+))$ by restriction, it follows immediately that
\begin{equation}
  \label{eq:fufu}
  F_{t,x}v_\varepsilon - F_{t,x}v = \sum_{k=1}^{m-1} \frac{1}{k!}
  F_{t,x}v_k \, \varepsilon^k + F_{t,x} R_{m,\varepsilon},
\end{equation}
as an identity in $L^p(\Omega)$. Similar considerations can be made
with $F$ in place of $F_{t,x}$.

An alternative way to reach the same conclusion is to look at the
composition of functions
\[
\varepsilon \longmapsto v_\varepsilon \longmapsto F_{t,x}v_\varepsilon,
\]
where $\varepsilon \mapsto v_\varepsilon$ is of class $C^{m-1}$ from
$\erre$ to $\sC^p$ and $F_{t,x}$ is of class $C^\infty$ from $\sC^p$
to $L^p(\Omega)$, so that $\varepsilon \mapsto F_{t,x}v_\varepsilon$
is of class $C^{m-1}$ from $\erre$ to $L^p(\Omega)$, and the series
expansion \eqref{eq:fufu} follows by Taylor's theorem.

To obtain a series expansion for the difference
$B_\varepsilon(t,x)-B(t,x)$ we need, however, to work pathwise,
i.e. in $L^0(\Omega)$, essentially because it seems difficult to find
a (reasonable) Banach space $E$ such that $x \mapsto e^{-x}$ is
Fr\'echet differentiable from $L^p(\Omega)$ to $E$, so that the chain
rule could be applied to obtain a differentiability result for the map
$\varepsilon \mapsto B_\varepsilon(t,x)$. We proceed instead as
follows: Taylor's theorem yields
\begin{align*}
e^{-x} &= 1 + \sum_{j=1}^{m-1} (-1)^j \frac{x^j}{j!} 
         + (-1)^m \int_0^1 \frac{(1-s)^{m-1}}{(m-1)!} e^{-sx} x^m\,ds\\
  & =: 1 + J_{m-1}(x) + r_m(x),
\end{align*}
hence
\begin{align*}
  B_\varepsilon(t,x) 
  &= \exp\bigl( -F_{t,x}v_\varepsilon \bigr)
    = \exp\bigl( -F_{t,x}v \bigr) \exp\bigl( -F_{t,x}(v_\varepsilon-v) \bigr)\\
  &= B(t,x)\Bigl(1 + J_{m-1}\bigl( F_{t,x}v_\varepsilon - F_{t,x}v \bigr)
    + r_m\bigl( F_{t,x}v_\varepsilon - F_{t,x}v \bigr) \Bigr),
\end{align*}
so that the relative pricing error can be written as
\begin{equation}
  \label{eq:BB}
  \frac{B_\varepsilon(t,x)-B(t,x)}{B(t,x)} = J_{m-1}\bigl(
  F_{t,x}v_\varepsilon - F_{t,x}v \bigr) + r_m\bigl(
  F_{t,x}v_\varepsilon - F_{t,x}v \bigr).
\end{equation}
Substituting the series expansion of $F_{t,x}v_\varepsilon - F_{t,x}v$
provided by \eqref{eq:fufu}, we obtain a series expansion of
$B_\varepsilon(t,x)-B(t,x)$ in $\varepsilon$ of order $m-1$ with a
rest of higher order, that has to be interpreted as an identity in
$L^0(\Omega)$.

Note that if $B_0 \in
L^{2mp}(\Omega;L^2(0,T;\cL^2(\ell^2;H(\erre_+))))$, then
$v_\varepsilon$ and $v_\varepsilon$ belong to $\sC^{mp}(H(\erre))$,
which implies that $v_1,\ldots,v_{m-1}$ and $R_{m,\varepsilon}$ in
\eqref{eq:fufu} belong to $\sC^{mp}(H(\erre))$, hence all powers of
$F_{t,x}v_\varepsilon - F_{t,x}v$ up to the exponent $m$ produce
series in $\varepsilon$ whose coefficients belong to $L^p(\Omega)$, by
H\"older's inequality. We can thus write the first term on the
right-hand side of \eqref{eq:BB} as a series of order $m-1$ with
coefficients in $L^p(\Omega)$, plus a remainder of higher
order. Estimating the second term on the right-hand side of
\eqref{eq:BB} requires further assumptions. Let $p'$ be the (H\"older)
conjugate exponent to $p$. We have
\begin{align*}
  &\norm[\big]{r_m\bigl( F_{t,x}v_\varepsilon - F_{t,x}v \bigr)}_{L^1(\Omega)}\\
  &\hspace{3em} \leq \int_0^1 \frac{(1-s)^{m-1}}{(m-1)!} \norm[\big]{%
    \exp\bigl( -s(F_{t,x}v_\varepsilon - F_{t,x}v) \bigr)}_{L^{p'}(\Omega)}
     \norm[\big]{(F_{t,x}v_\varepsilon - F_{t,x}v)^m}_{L^p(\Omega)} \,ds,
\end{align*}
where $(Fv_\varepsilon - Fv)^m$ is a series of order $m$ or higher
with coefficients in $L^p$, hence its $L^p(\Omega)$ norm is bounded.
Since, by H\"older's inequality,
\begin{align*}
  \norm[\big]{ \exp\bigl( -s(F_{t,x}v_\varepsilon - F_{t,x}v)
  \bigr)}_{L^{p'}(\Omega)}
  &= \Bigl( \E \exp\bigl( -p's(Fv_\varepsilon-Fv) \bigr) \Bigr)^{1/p'}\\
  &\leq \Bigl( \E \exp\bigl( -p'(Fv_\varepsilon-Fv) \bigr) \Bigr)^{1/p'}
    \qquad \forall s \in [0,1],
\end{align*}
it follows that if $\E\exp \bigl( -p'(Fu_\varepsilon-Fu) \bigr)$ is
bounded for $\varepsilon$ in a (right) neighborhood of zero, we have a
series expansion of the relative pricing error with coefficients in
$L^1(\Omega)$.
An alternative estimate can be obtained under a sign assumption,
starting from the expression
\begin{align*}
  &B(t,x) \, r_m\bigl( F_{t,x}v_\varepsilon - F_{t,x}v \bigr)\\
  &\hspace{3em} = \int_0^1 \frac{(1-s)^{m-1}}{(m-1)!} \exp\bigl(
    -sF_{t,x}v_\varepsilon - (1-s)F_{t,x}v \bigr) (F_{t,x}v_\varepsilon
    - F_{t,x}v)^m \,ds.
\end{align*}
If $v \geq 0$ and $v_\varepsilon \geq 0$, then
\[
  \exp\bigl( -sF_{t,x}v_\varepsilon - (1-s)F_{t,x}v \bigr) \leq 1
  \qquad \forall s \in [0,1]
\]
because $F_{t,x}$ is positivity preserving. This implies
\[
  \norm[\big]{B(t,x) \, r_m\bigl( F_{t,x}v_\varepsilon - F_{t,x}v
    \bigr)}_{L^p(\Omega)} \lesssim \norm[\big]{( F_{t,x}v_\varepsilon
    - F_{t,x}v)^m}_{L^p(\Omega)},
\]
where the right hand side converges to zero in $L^p(\Omega)$ faster
than $\varepsilon^{m-1}$. Conditions ensuring the positivity of mild
solutions to the Musiela SPDE are discussed, e.g., in
\cite{FTT,Kusu:Mus}, and in \cite{cm:pos1,cm:pos2} in a more general
context.


\bibliographystyle{amsplain}
\bibliography{ref,finanza}

\end{document}